\documentclass[reqno, 12pt]{amsart}
\usepackage{amsmath,amsxtra,amssymb,latexsym, amscd,amsthm}
\usepackage[mathscr]{euscript}
\usepackage{mathrsfs}
\usepackage{color}
\usepackage[active]{srcltx}
\usepackage{mathtools,stmaryrd}
\usepackage{enumerate}
\usepackage{hyperref}
\usepackage{framed}
\usepackage[T1]{fontenc}

\newcommand{\RR}{{\mathbb R}}

\newcommand{\N}{\mathbb{N}}

\newcommand{\bx}{x}

\newcommand{\by}{y}
\newcommand{\K}{\mathbf{K}}
\newcommand{\Y}{\mathbf{Y}}
\newcommand{\F}{\mathbf{F}}
\newcommand{\bs}{\mathbf{S}}
\newcommand{\bX}{x}
\newcommand{\bY}{y}
\newcommand{\ud}{\mathrm{d}}
\newcommand{\td}[1]{\tilde{#1}}

\newcommand{\rank}{\mbox{\upshape rank}}

\newcommand{\psdp}{\mbox{\upshape\tiny primal}}
\newcommand{\dsdp}{\mbox{\upshape\tiny dual}}

\newcommand{\re}{\mathbb{R}}
\newcommand{\csk}[1]{{\mathcal{C}_{#1}}[x]}
\newcommand{\cs}{\mathcal{C}[x]}
\newcommand{\cyk}[1]{\mathcal{C}_{#1}[y]}
\newcommand{\cy}{\mathcal{C}[y]}
\def\af{\alpha}

\newcommand{\mH}{\mathscr{H}}
\newcommand{\mL}{\mathscr{L}}
\newcommand{\wt}[1]{\widetilde{#1}}
\newcommand{\epi}{\text{\upshape epi}}

\newcommand{\co}{\text{\upshape co}}
\newcommand{\cl}{\text{\upshape cl}}

\newcommand{\qm}{\mathbf{qmodule}}

\newtheorem{algorithm}{Algorithm}[section]

\newtheorem{condition}{Condition}[section]
\newtheorem{theorem}{Theorem}[section]

\newtheorem{lemma}{Lemma}[section]
\newtheorem{example}{Example}[section]
\newtheorem{definition}{Definition}[section]
\newtheorem{remark}{Remark}[section]
\newtheorem{proposition}{Proposition}[section]

\def\EES{{\accent"5E e}\kern-.5em\raise.8ex\hbox{\char'23 }}
\def\ow{o\kern-.42em\raise.82ex\hbox{
   \vrule width .12em height .0ex depth .075ex \kern-0.16em \char'56}\kern-.07em}
\def\OW{o\kern-.460em\raise1.36ex\hbox{
\vrule width .13em height .0ex depth .075ex \kern-0.16em
\char'56}\kern-.07em}
\def\DD{D\kern-.7em\raise0.4ex\hbox{\char '55}\kern.33em}
\title{On Solving a Class of Fractional Semi-infinite Polynomial Programming Problems}

\author{Feng Guo}
\address[Feng Guo]{School of Mathematical Sciences, Dalian  University of Technology, Dalian, 116024, China}
\email{fguo@dlut.edu.cn}

\author{Liguo Jiao$^*$}
\address[Liguo Jiao]{School of Mathematical Sciences, Soochow University, Suzhou 215006, Jiangsu Province, China}
\email{hanchezi@163.com}
\thanks{$^*$Corresponding Author.}
\date{ \today}

\begin{document}

\begin{abstract}
In this paper, we study a class of fractional
semi-infinite polynomial programming (FSIPP) problems, in which
the objective is a fraction of a convex polynomial and a concave
polynomial, and the constraints consist of
infinitely many convex polynomial inequalities.
To solve such a problem, we first 
reformulate it to a pair of primal and dual conic optimization
problems, which reduce to semidefinite programming (SDP)
problems if we can bring sum-of-squares structures into the conic
constraints.
To this end, we provide a characteristic cone constraint qualification
for convex
semi-infinite programming problems to guarantee strong duality and
also the attainment of the solution in the dual problem, which is of
its own interest. 
In this framework, we first present a hierarchy of SDP relaxations with
asymptotic convergence for the FSIPP problem whose index set is
defined by finitely many polynomial inequalities. 
Next, we study four cases of the FSIPP problems which can be reduced to
either a single SDP problem or a finite sequence of SDP problems,
where at least one minimizer can be extracted. Then, 
we apply this approach to the four corresponding multi-objective cases
to find efficient solutions.
\end{abstract}
\subjclass[2010]{65K05; 90C22; 90C29; 90C34}
\keywords{fractional optimization;
convex semi-infinite systems; semidefinite programming relaxations;
sum-of-squares; polynomial optimization}

\maketitle

\section{Introduction}

In this paper, we consider the fractional semi-infinite polynomial
programming (FSIPP) problem in the following form:
\begin{equation}\label{FMP}
	\left\{
	\begin{aligned}
		r^{\star}:=\min_{x\in\RR^m}&\  \frac{f(x)}{g(x)}\\
		\text{s.t.}&\ \psi_1(x)\le 0,\ldots,\psi_s(x)\le 0, \\
		&\ p(\bx,\by)\le 0,\ \ \forall \by\in \Y\subset\RR^n,
	\end{aligned}\right.
\end{equation}
where $f, g$, $\psi_1,\ldots,\psi_s\in\RR[\bX]$ and
$p\in\RR[\bX,\bY]$. Here, $\RR[x]$ (resp. $\RR[x,y]$) denotes the ring
of real polynomials in $x=(x_1,\ldots,x_m)$ (resp., $x=(x_1,\ldots,x_m)$
and $y=(y_1,\ldots,y_n)$).
%with real coefficients.
Denote by $\K$ and $\bs$ the feasible set and the optimal solution set
of \eqref{FMP}, respectively.
In this paper, we assume that $\bs\neq\emptyset$ and make the
following assumptions on \eqref{FMP}:
\begin{quote}
		{\sf (A1):}\ $\Y$ is compact;
		$f$, $-g$, each $\psi_i$ and $p(\cdot, \by)$ for each $y\in
		\Y$ are all convex in $x$;\\
		{\sf (A2):}\ Either $f(x)\ge 0$ and $g(x)>0$ for all $x\in\K$; or $g(x)$
		is affine and $g(x)>0$ for all $x\in\K.$
\end{quote}
The feasible set $\K$ is convex by {\sf (A1)}, while
	the objective of \eqref{FMP} is generally {\it not} convex.
	The assumption {\sf (A2)} ensures that the objective of
	\eqref{FMP} is well-defined. It is commonly adopted in the literature
about fractional optimization problems
\cite{Jeyakumar2012,Nguyen2016,Schaible2003}, and can be satisfied by
many practical optimization models
\cite{Bajalinov2003,Nguyen2016,Stancu-Minasian1997}.
If $\Y$ is noncompact, the technique of homogenization can be
applied (c.f. \cite{SIPPSDP}).

Over the last several decades, due to a great number of applications
in many fields, semi-infinite programming (SIP) has attracted a great
interest and has been a very active research area
\cite{Goberna2017,Goberna2018,Hettich1993,Still2007}.
Numerically, SIP problems can be solved by different approaches
including, for instance, discretization methods,
local reduction methods, exchange methods,
simplex-like methods, feasible point methods etc; see
\cite{Goberna2017,Hettich1993,Still2007}
and the references therein for details.
A main difficulties in solving general SIP problems
is that the feasibility test of a given point is equivalent to
{\itshape globally} solving a lower level subproblem which is
generally nonlinear and nonconvex.

To the best of our knowledge, there are only limited research results devoted
to semi-infinite polynomial optimization by  exploiting features of
{\it polynomial} optimization problems. For instance, Parpas and Rustem \cite{PR2009}
proposed a discretization-like method to solve minimax polynomial optimization
problems, which can be reformulated as semi-infinite polynomial
programming (SIPP) problems.
Using polynomial approximation and an appropriate hierarchy of semidefinite
programming (SDP) relaxations, Lasserre \cite{Lasserre2011} presented an
algorithm to solve the generalized SIPP problems. Based on an exchange
scheme, an SDP relaxation method for solving
SIPP problems was proposed in \cite{SIPPSDP}.
By using representations of nonnegative polynomials in the
univariate case, an SDP method was given in \cite{XSQ}
for linear SIPP problems with the index set being closed intervals.
For convex SIPP problems, Guo and Sun \cite{GS2020} proposed an SDP
relaxation method by combining the sum-of-squares representation of
the Lagrangian
function with high degree perturbations \cite{Lasserre05} and Putinar's
representation \cite{Putinar1993}
of the constraint polynomial on the index set.

In this paper, in a way similar to \cite{GS2020}, we will derive an
SDP relaxation method for the FSIPP problem \eqref{FMP}.
Different from \cite{GS2020}, we treat the problem in a more
systematical manner. We first reformulate the FSIPP problem to a conic
optimization problem and its Lagrangian dual, which involve two
convex cones of polynomials in the variables and the parameters,
respectively (Section \ref{sec::SR}). Under suitable assumptions on
these cones, approximate minimum and minimizers of the FSIPP problem
can be obtained by solving the conic reformulations. If we can 
bring sum-of-squares structures into these cones, the conic
reformulations reduce to a pair of SDP problems and become tractable. 
To this end, inspired by Jeyakumar and Li~\cite{Jeyakumar2010},
we provide a characteristic cone constraint qualification for convex
semi-infinite programming problems to guarantee the strong duality and
the attachment of the solution in the dual problem, which is of its
own interest. We remark that this constraint qualification, which is crucial for
some applications
(see Section \ref{sec::mfsipp}), is weaker
than the Slater condition used in \cite{GS2020}.

In what follows, we first present a hierarchy of SDP relaxations for
the FSIPP problem whose index set is defined by finitely many
polynomial inequalities (Section \ref{sec::SDP}). This
is done by introducing appropriate quadratic modules to the conic reformulations.
The asymptotic convergence of the optimal values of the SDP
relaxations to the minimum of the FSIPP problem can be established by
Putinar's Positivstellensatz \cite{Putinar1993}.
Moreover, when the FSIPP problem has a unique minimizer, this
minimizer can be approximated from the optimal solutions to the SDP
relaxations. By means of existing complexity results of Putinar's
Positivstellensatz, we can derive some convergence rate analysis of
the SDP relaxations. 
We also present some discussions on the stop criterion for such SDP
relaxations.
Next, we restrict our focus
on four cases of FSIPP problems for which the SDP relaxation method is exact or
has finite convergence and at least one minimizer can be extracted
(Section \ref{sec::SDP4}). 
The reason for this restriction is due to some applications of the
FSIPP problem where exact minimum and minimizers are required.
In particular, we study the application of our SDP method to the
corresponding multi-objective FSIPP problems, 
where the objective function is a vector valued function
with each component being a fractional function. We
aim to find efficient solutions (see Definition \ref{def::es}) to such
problems.
Some sufficient efficiency conditions and duality results for such
problems can be found in \cite{Chuong2016,Verma2017,ZZ1,ZZ2}. However,
as far as we know, very few algorithmic developments
are available for such problems in the literature
because of the difficulty of checking feasibility of a given point.

This paper is organized as follows. Some notation and preliminaries
are given in Section \ref{sec::pre}. 
The FSIPP problem is reformulated as a conic optimization problem in
Section \ref{sec::SR}. We present a hierarchy of SDP
relaxations for FSIPP problems in Section \ref{sec::SDP}. In Section
\ref{sec::SDP4}, we consider four specified cases of
the FSIPP problems and the application of
the SDP method to the multi-objective cases. 
Some conclusions are given in Section \ref{sec::conclusions}.

\section{Notation and Preliminaries}\label{sec::pre}

In this section, we collect some  notation and preliminary results
which will be used in this paper. 
The symbol $\N$ (resp., $\re$, $\RR_+$) denotes
the set of nonnegative integers (resp., real numbers, nonnegative real
numbers). For any $t\in \re$, $\lceil t\rceil$ denotes the smallest
integer that is not smaller than $t$.
For $u\in \re^m$,
$\Vert u\Vert$ denotes the standard Euclidean norm of $u$.
For $\alpha=(\alpha_1,\ldots,\alpha_n)\in\N^n$,
$|\alpha|=\alpha_1+\cdots+\alpha_n$.
For $k\in\N$, denote $\N^n_k=\{\alpha\in\N^n\mid |\alpha|\le
k\}$ and $\vert\N^n_k\vert$ its cardinality.
For variables $x \in \re^m$, $y \in \re^n$ and
$\beta\in\N^m, \af \in \N^n$, $x^\beta$, $y^\af$ denote
$x_1^{\beta_1}\cdots x_m^{\beta_m}$,
$y_1^{\af_1}\cdots y_n^{\af_n}$, respectively.
For $h \in \RR[x]$, we denote by $\deg (h)$ its (total) degree.
For $k\in\N$, denote by $\RR[x]_k$ (resp., $\RR[y]_k$) the set of polynomials
in $\RR[x]$ (resp., $\RR[y]$) of degree up to $k$.
For $A=\RR[x],\ \RR[y],\ \RR[x]_k,\ \RR[y]_k$, denote by $A^*$ the dual
space of linear functionals from $A$ to $\RR$.

\vskip 5pt

A polynomial $h\in\RR[x]$ is said to be a
sum-of-squares (s.o.s) of polynomials if it can be written as $h=\sum_{i=1}^l
h_i^2$ for some $h_1,\ldots,h_l \in\RR[x]$.
The symbols $\Sigma^2[x]$ and $\Sigma^2[y]$ denote the sets of polynomials that are
s.o.s in $\RR[x]$ and $\RR[y]$, respectively.
Notice that not every nonnegative polynomial can be written as s.o.s,
see \cite{Reznick96someconcrete}. We recall the following properties
about polynomials nonnegative on certain sets, which will be used in
this paper.
\begin{theorem}[Hilbert's theorem]\label{th::hilbert}
	Every nonnegative polynomial $h\in\RR[x]$ can be written as s.o.s
	in the following cases: {\rm (i) $m=1$; (ii) $\deg(h)=2$; (iii) $m=2$}
	and $\deg(h)=4$.
\end{theorem}
\begin{theorem}[The $S$-lemma]\label{th::slemma}
	Let $h,$ $q\in\RR[x]$ be two quadratic polynomials and assume that there
	exists $u^0\in\RR^m$ with $q(u^0)>0$. The following assertions are
	equivalent: {\rm (i) $q(x)\ge 0$ $\Rightarrow$ $h(x)\ge 0$;} {\rm (ii)} there
	exists $\lambda\ge 0$ such that $h(x)\ge \lambda q(x)$ for all
	$x\in\RR^m$.
\end{theorem}

\begin{proposition}{\rm\cite[Example 3.18]{ScheidererCurve}}\label{prop::s}
	Let $q\in\RR[x]$ and the set $\mathcal{K}=\{x\in\RR^m\mid q(x)\ge
	0\}$ be compact. If $h \in \RR[x]$ is nonnegative on $\mathcal{K}$ and the
	following conditions hold
	\begin{enumerate}[\upshape (i)]
		\item $h$ has only finitely many zeros on $\mathcal{K},$
			each lying in the interior of $\mathcal{K};$
		\item the Hessian $\nabla^2 h$ is positive definite on each of
			these zeros$,$
	\end{enumerate}
	then $h=\sigma_0+\sigma_1q$ for some $\sigma_0,
	\sigma_1\in\Sigma^2[x]$.
\end{proposition}

For $h, h_1, \ldots, h_\kappa\in\RR[x]$, 
let us recall some background about Lasserre's hierarchy
\cite{LasserreGlobal2001} for the polynomial optimization problem
\begin{equation}\label{eq::PO}
		h^{\star}:=\min_{\bx\in\RR^m}\ h(x)\quad \text{s.t.}\quad
		x\in S:=\{x\in\RR^m \mid h_1(x)\ge 0, \ldots, h_{\kappa}(x)\ge 0\}\not= \emptyset.
\end{equation}
Let $H:=\{h_1,\ldots,h_{\kappa}\}$ and $h_0=1$ for convenience.  
We denote by 
\[
  \qm(H):=\left\{\sum_{j=0}^\kappa h_j\sigma_j\ \Big|\ 
\sigma_j \in \Sigma^2[x], j=0,1,\ldots,\kappa \right\}
\]
the {\itshape quadratic module} generated by $H$ and denote by 
\[
	\qm_k(H):=\left\{\sum_{j=0}^\kappa h_j\sigma_j\ \Big|\ 
\sigma_j \in \Sigma^2[x], \,
\deg(h_j\sigma_j)\le 2k,\ j=0,1,\ldots,\kappa\right\}
\]
its  $k$-th {\itshape quadratic module}.
It is clear that if $h\in\qm(H)$, then
$h(x)\ge 0$ for any $x\in S$. However, the converse is not necessarily
true.

\begin{definition}\label{def::AC}
	We say that $\qm(H)$
is {\itshape Archimedean}  if
there exists $\phi\in \qm(H)$ such that the inequality $\phi(x)\ge 0$
defines a compact set in $\RR^m;$ or equivalently$,$ if for all
$\phi\in\RR[x],$ there is some $N\in\N$ such that $N\pm
\phi\in\qm(H)$ $($c.f.$,$ \cite{Kmoment}$)$. 
\end{definition}

Note that for any compact set $S$ we can always
force the associated quadratic module to be Archimedean
by adding a
redundant constraint $M-\Vert x\Vert^2_2\ge 0$ in the description of
$S$ for a sufficiently large $M$.

\begin{theorem}\label{th::PP}{\upshape\cite[{Putinar's Positivstellensatz\/}]{Putinar1993}}
Suppose that $\qm(H)$ is Archimedean. If a polynomial $\psi\in\RR[x]$ is
		positive on $S,$ then $\psi\in\qm_k(H)$ for some $k\in\N$. 
\end{theorem}

For a polynomial $\psi(x)=\sum_{\alpha\in\N^m} \psi_\alpha
x^\alpha\in\RR[x]$ where $\psi_\alpha$ denotes the coefficient of the
monomial $x^\alpha$ in $\psi$, define
the norm
\begin{equation}\label{eq::fnorm}
	\Vert \psi\Vert:=\max_{\alpha\in\N^m}\frac{\vert
	\psi_\alpha\vert}{\tbinom{|\alpha|}{\alpha}}.
\end{equation}

We have the following result for an estimation of the order $k$ in
Theorem \ref{th::PP}. 
\begin{theorem}{\upshape \cite[Theorem 6]{NieSchweighofer}}\label{th::complexity}
Suppose that $\qm(H)$ is Archimedean and $S\subseteq
(-\tau_S,\tau_S)^m$ for some $\tau_S>0$. Then there is some
positive $c\in\RR$ $($depending only on $h_j$'s$)$ such that for all
$\psi\in\RR[x]$ of degree $d$ with $\min_{x\in S}\psi(x)>0,$ we have
$\psi\in\qm_k(H)$ whenever
\[
	k\ge c\exp\left[\left(d^2m^d\frac{\Vert \psi\Vert\tau_S^d}{\min_{x\in
	S}\psi(x)}\right)^c\right]. 
\]
\end{theorem}

For an integer $k\ge \max\{\lceil\deg(h)/2\rceil, k_0\}$ where 
$k_0:=\max\{\lceil\deg(h_i)/2\rceil, i=1,\ldots,\kappa\}$, 
the $k$-th Lasserre's relaxation for
\eqref{eq::PO} is
\begin{equation}\label{eq::lppp}
		h_k^{\dsdp}:=\inf_{\mL}\ \mL(h)\ \ \text{s.t.}\ \ 
		\mL\in(\qm_k(H))^*,\ \mL(1)=1,
\end{equation}
and its dual problem is
\begin{equation}\label{eq::lppd}
		h_k^{\psdp}:=\sup_{\rho}\ \rho\ \
		\text{s.t.}\ \ h-\rho\in\qm_k(H). 
\end{equation}
For each $k\ge k_0$, \eqref{eq::lppp} and \eqref{eq::lppd} can be
reduced to a pair of primal and dual SDP problems, and we always have
$h_k^{\psdp}\le h_k^{\dsdp}\le h^{\star}$ (c.f. \cite{LasserreGlobal2001}).
The convergence of $h_k^{\dsdp}$ and $h_k^{\psdp}$ to $h^{\star}$ as
$k\rightarrow\infty$ can
be established by Putinar's Positivstellensatz
\cite{LasserreGlobal2001,Putinar1993} under the Archimedean condition. If
there exists an integer $k^{\star}$ such that $h_k^{\psdp}=h_k^{\dsdp}=h^{\star}$
for each $k\ge k^{\star}$, we say that Lasserre's hierarchy for
\eqref{eq::PO}
has finite convergence. 
To certify $h_k^{\dsdp}=h^{\star}$
when it occurs, a sufficient condition on a minimizer of
\eqref{eq::lppp} called flat extension condition \cite{CFKmoment} is
available. A weaker
condition called flat truncation condition was
proposed by Nie in \cite{NieFlatTruncation}. Precisely,
for a linear functional $\mL\in(\RR[x]_{2k})^*$,
denote by $\mathbf{M}_k(\mL)$ the associated $k$-th moment matrix which is
indexed by $\mathbb{N}^m_{k}$, with
$(\alpha,\beta)$-th entry being $\mL(x^{\alpha+\beta})$ for $\alpha, \beta \in
\mathbb{N}^m_{k}$.
\begin{condition}{\upshape\cite[flat truncation
	condition]{NieFlatTruncation}}\label{con::truncation}
	A minimizer $\mL^{\star}\in(\RR[x]_{2k})^*$ of \eqref{eq::lppp}
	satisfies the following rank condition$:$ there exists an integer
	$k'$ with 
	$\max\{\lceil\deg(h)/2\rceil, k_0\}\le k'\le  k$ such that
	\[
		\rank\ \mathbf{M}_{k'-k_0}(\mL^{\star})=\rank\
		\mathbf{M}_{k'}(\mL^{\star}).
	\]
\end{condition}
Nie \cite[Theorem 2.2]{NieFlatTruncation} proved that Lasserre's
hierarchy has
finite convergence if and only if the flat truncation holds, under some
generic assumptions. In particular, we have the following proposition.
\begin{proposition}{\upshape \cite[c.f. Theorem
	2.2]{NieFlatTruncation}}\label{prop::fc}
	Suppose that the set $\{x\in\RR^m\mid h(x)=h^{\star}\}$ is finite$,$
	the set of global minimizers of \eqref{eq::PO} is nonempty$,$ and for
	$k$ big enough the optimal value of \eqref{eq::lppd} is achievable
	and there is no duality gap between \eqref{eq::lppp} and
	\eqref{eq::lppd}. Then$,$ for $k$ sufficiently large$,$
	$h_k^{\psdp}=h_k^{\dsdp}=h^{\star}$ if and only if every minimizer of
	\eqref{eq::lppp} satisfies the flat truncation condition.
\end{proposition}
Moreover, if a minimizer $\mL^{\star}\in(\RR[x]_{2k})^*$ of
\eqref{eq::PO} satisfies the flat extension condition or the flat
truncation condition, we can extract finitely many global minimizers of
\eqref{eq::PO} from the moment matrix $\mathbf{M}_k(\mL^{\star})$ by
solving a linear algebra problem (c.f. \cite{CFKmoment,LasserreHenrion}). 
Such procedure has been implemented in the software {\sf GloptiPoly}
\cite{gloptipoly} developed by Henrion, Lasserre and L{\"o}fberg.

\vskip 8pt

We say that a linear functional $\mL\in(\RR[x])^*$ has a representing
measure $\mu$ if
there exists a Borel measure $\mu$ on $\RR^m$ such that
\[
	\mL(x^\alpha)=\int_{\RR^m} x^{\alpha}\ud\mu(x),\quad \forall \alpha\in\N^m.
\]
For $k\in\N$, we say that $\mL\in(\RR[x]_k)^*$ has a representing measure
$\mu$ if the above holds for all $\alpha\in\N^m_{k}$.
By \cite[Theorem 1.1]{CFKmoment}, Condition
\ref{con::truncation} implies that $\mL^{\star}$ has an atomic representing
measure. Each atom of the measure is a global minimizer of
\eqref{eq::PO} and can be extracted by the procedure presented in
\cite{LasserreHenrion}.
\vskip 5pt

We end this section by introducing some important  properties on convex
polynomials.
\begin{definition}
	A polynomial $h\in\RR[x]$ is {\it coercive} whenever the lower level
	set $\{x\in\RR^m\mid h(x)\le \alpha\}$ is a $($possibly empty$)$
	compact set$,$ for all $\alpha\in\RR$.
\end{definition}
\begin{proposition}{\upshape \cite[Lemma
	3.1]{JEYAKUMAR201434}}\label{prop::sc}
	Let $h\in\RR[x]$ be a convex polynomial. If $\nabla^2h(u')\succ
	0$ at some point $u'\in\RR^m,$ then $h$ is coercive and strictly
	convex on $\RR^m.$
\end{proposition}

Recall also a subclass of convex polynomials in $\RR[x]$ introduced
by Helton and Nie \cite{SRCSHN}.
\begin{definition}{\upshape \cite{SRCSHN}}\label{sos-convex}
	A polynomial $h\in\RR[x]$ is s.o.s-convex if its Hessian $\nabla^2
	h$ is an s.o.s-matrix$,$ i.e.$,$ there is some integer $r$ and some matrix
	polynomial $H\in\RR[x]^{r\times m}$ such that $\nabla^2
	h(x)=H(x)^TH(x)$.
\end{definition}

In fact, Ahmadi and Parrilo \cite{Ahmadi2012} has proved that a convex polynomial
$h\in\RR[x]$ is s.o.s-convex if and only if $m=1$ or $\deg(h)=2$ or
$(m,\deg(h))=(2,4)$. In particular, the class of s.o.s-convex
polynomials contains the
classes of separable convex polynomials and convex quadratic
functions.

The significance of s.o.s-convexity is that
it can be checked numerically by solving an SDP problem (see
\cite{SRCSHN}), while checking the convexity of a polynomial is generally
NP-hard (c.f. \cite{Ahmadi2013}).
Interestingly, an extended Jensen's inequality holds for s.o.s-convex
polynomials.
\begin{proposition}{\upshape\cite[Theorem 2.6]{LasserreConvex}}\label{prop::Jensen}
	Let $h\in\RR[x]_{2d}$ be s.o.s-convex$,$ and let
	$\mL\in(\RR[x]_{2d})^*$ satisfy $\mL(1)=1$ and $\mL(\sigma)\ge 0$ for
	every $\sigma\in\Sigma^2[x]\cap\RR[x]_{2d}$. Then$,$ $$\mL(h(x))\ge
	h(\mL(x_1),\ldots,\mL(x_m)).$$
\end{proposition}

The following result plays a significant role for the SDP
relaxations of \eqref{FMP} in Section \ref{sec::SDP4}.
\begin{lemma}{\upshape\cite[Lemma 8]{SRCSHN}}\label{lem::sosconvex}
	Let $h\in\RR[x]$ be s.o.s-convex. If $h(u)=0$ and $\nabla h(u)=0$
	for some $u\in\RR^m,$ then $h$ is an s.o.s polynonmial.
\end{lemma}

\section{Conic reformulation of FSIPP}\label{sec::SR}

In this section, we first reformulate the FSIPP problem \eqref{FMP} to
a conic optimization problem and its Lagrangian dual, which involve
two convex subcones of $\RR[x]$ and $\RR[y]$, respectively. Under
suitable assumptions on
these cones, we show that approximate minimum and minimizers of
\eqref{FMP} can be obtained by solving the conic reformulations.

\subsection{Problem Reformulation}
Consider the problem
\begin{equation}\label{FMP2}
	\min_{x\in\K}\ \  f(x)-r^{\star}g(x).
\end{equation}
Note that, under {\sf (A1-2)}, \eqref{FMP2} is clearly a  convex
semi-infinite programming problem and its optimal value is $0$.

Now we consider the Lagrangian dual of \eqref{FMP2}. In particular,
the constraints $p(x,y)\le 0$ for all $y\in\Y$ means that for each
$u\in\K$, $p(u,y)\in\RR[y]$ belongs to the cone of nonpositive
polynomials on $\Y$. The polar of this cone is taken as the cone of
finite nonnegative measures supported on $\Y$, which we denote by
$\mathcal{M}(\Y)$.
Therefore, we can define the Lagrangian function $L_{f,g}(x,\mu,\eta):
\RR^m\times\mathcal{M}(\Y)\times\RR_+^s \rightarrow \RR$ by 
\begin{equation}\label{eq::lag}
		L_{f,g}(x,\mu,\eta)=f(x)-r^{\star}g(x)+\int_{\Y}
		p(x,y)d\mu(y)+\sum_{j=1}^s\eta_j\psi_j(x).
\end{equation}
Then, the Lagrangian dual of \eqref{FMP2} reads
\begin{equation}\label{eq::dcsip}
	\max_{\mu\in\mathcal{M}(\Y),\eta_j\ge 0} \inf_{x\in\RR^m}
	L_{f,g}(x,\mu,\eta). 
\end{equation}
See \cite{HaarDual,Hettich1993,Still2007,ShapiroSIP} for more details.

Consider the strong duality and dual attainment for the dual pair
\eqref{FMP2} and \eqref{eq::dcsip}:
\begin{quote}
	{\sf (A3):}
		$\exists\mu^{\star}\in\mathcal{M}(\Y)$ and $\eta^{\star}\in\RR_+^s$ such
		that $\inf_{x\in\RR^m} L_{f,g}(x,\mu^{\star},\eta^{\star})=0$.
\end{quote}
The following straightforward result is essential for the SDP
relaxations of \eqref{FMP}
in Section \ref{sec::SDP4}.

\begin{proposition}\label{prop::red}
	Under {\sf (A1-3)}$,$
	$L_{f,g}(u^{\star},\mu^{\star},\eta^{\star})=0$ and $\nabla_x
	L_{f,g}(u^{\star},\mu^{\star},\eta^{\star})=0$ for any $u^{\star}\in\bs$.
\end{proposition}

Under {\sf (A1-3)}, we first reformulate \eqref{FMP} to a conic
optimization problem with the
same optimal value $r^{\star}$. We need the following notation.
For a subset $\mathcal{X}\subset\RR^m$ (resp., the index set $\Y$),
denote by $\mathscr{P}(\mathcal{X})\subset\RR[x]$ (resp.
$\mathscr{P}(\Y)\subset\RR[y]$) the cone of
nonnegative polynomials on $\mathcal{X}$ (resp., $\Y$).
For $\mL\in(\RR[x])^*$ (resp., $\mH\in(\RR[y])^*$), denote by
$\mL(p(x,y))$ (resp., $\mH(p(x,y))$) the image of $\mL$ (resp., $\mH$)
on $p(x,y)$ regarded as an element in $\RR[x]$ (resp., $\RR[y]$) with
coefficients in $\RR[y]$ (resp., $\RR[x]$), i.e.,
$\mL(p(x,y))\in\RR[y]$ (resp., $\mH(p(x,y)))\in\RR[x]$).
For a subset $\mathcal{X}\subset\RR^m$, consider the conic
optimization problem
\begin{equation}\label{eq::form1}
\left\{
	\begin{aligned}
		\td{r}:=\sup_{\rho,\mH,\eta}\
		&\rho&\\ \text{s.t.}\
		&f(\bX)-\rho
		g(x)+\mH(p(\bX,\by))+\sum_{j=1}^s\eta_j\psi_j(x)\in\mathscr{P}(\mathcal{X}),
		&\\
		&\rho\in\RR,\ \mH\in(\mathscr{P}(\Y))^*,\ \eta\in\RR_+^s.&
	\end{aligned}\right.
\end{equation}
\begin{proposition}\label{prop::eq}
Under {\sf (A1-3)}, suppose that $\mathcal{X}\cap\bs\neq\emptyset$,
then $\td{r}=r^{\star}$. 
\end{proposition}
\begin{proof}
Under {\sf (A3)}, define $\mH^{\star}\in(\RR[y])^*$ by letting
$\mH^{\star}(y^\beta)=\int_{\Y} y^\beta d\mu^{\star}(\by)$ for any
$\beta\in\N^n$. Then, $\mH^{\star}\in(\mathscr{P}(\Y))^*$ and 
$(r^{\star},\mH^{\star},\eta^{\star})$ is feasible to
\eqref{eq::form1}. Hence, $\td{r}\ge r^{\star}$. As
$\mathcal{X}\cap\bs\neq\emptyset$, for any
$u^{\star}\in\mathcal{X}\cap\bs$ and any
$(\rho,\mH,\eta)$ feasible to \eqref{eq::form1}, it holds that
\[
	f(u^{\star})-\rho
	g(u^{\star})+\mH(p(u^{\star},\by))+\sum_{j=1}^s\eta_j\psi_j(u^{\star})\ge
	0.
\]
Then, the feasibility of $u^{\star}$ to \eqref{FMP} implies that
$r^{\star}=\frac{f(u^{\star})}{g(u^{\star})}\ge\rho$ and thus
$r^{\star}\ge \td{r}$. 
\end{proof}

\begin{remark}\label{rk::eq}{\rm
Before moving forward, let us give a brief overview on the strategies
we adopted in this paper to construct SDP relaxations of
\eqref{FMP} from the reformulation
\eqref{eq::form1}. The difficulty of \eqref{eq::form1} is that the
exact representions of the convex cones $\mathscr{P}(\mathcal{X})$ and
$\mathscr{P}(\Y)$ are usually not available in general case, which makes
\eqref{eq::form1} still intractable. However, the Positivstellensatz from
real algebraic geometry, which provides algebraic certificates for
positivity (nonnegative) of polynomials on semialgebraic sets, can
gives us approximations of $\mathscr{P}(\mathcal{X})$ (with a
carefully chosen $\mathcal{X}$) and $\mathscr{P}(\Y)$
with sum-of-squares structures. Replacing $\mathscr{P}(\mathcal{X})$
and $\mathscr{P}(\Y)$ by these approximations, we can convert
\eqref{eq::form1} into SDP relaxations of \eqref{FMP}. In order to
obtain (approximate) optimum and optimizers of \eqref{FMP} from the
resulting SDP relaxations, we need establish some conditions which
should be satisfied by these approximations of
$\mathscr{P}(\mathcal{X})$ and $\mathscr{P}(\Y)$. 
Then according to these conditions, we can choose suitable subset
$\mathcal{X}$ and construct appropriate approximations
$\mathscr{P}(\mathcal{X})$ and $\mathscr{P}(\Y)$. Remark that
different subsets
$\mathcal{X}$ and versions of Positivstellensatz can lead to different
approximations of $\mathscr{P}(\mathcal{X})$ and $\mathscr{P}(\Y)$
satisfying the established conditions. Therefore, to present our approach in
a unified way, we next use the symbols $\cs$ and $\cy$ to denote
approximations of $\mathscr{P}(\mathcal{X})$ and $\mathscr{P}(\Y)$,
respectively, and derive the conditions they should satisfy (Theorem
\ref{th::main}). Then,
we specify suitable $\cs$ and $\cy$ in different situations to construct
concrete SDP relaxations of \eqref{FMP} in Section \ref{sec::SDP} and
\ref{sec::SDP4}. 
\qed
}
\end{remark}

Let $\cs$ (resp., $\cy$) be a convex cone in $\RR[x]$ (resp.,
$\RR[y]$). Replacing $\mathscr{P}(\mathcal{X})$ and $\mathscr{P}(\Y)$
by $\cs$ and $\cy$ in \eqref{eq::form1}, respectively, we get
the conic optimization problem
\begin{equation}\label{eq::f*r}
\left\{
	\begin{aligned}
		r^{\psdp}:=\sup_{\rho,\mH,\eta}\
		&\rho&\\
		\text{s.t.}\
		&f(\bX)-\rho
		g(x)+\mH(p(\bX,\by))+\sum_{j=1}^s\eta_j\psi_j(x)\in\cs,
		&\\
		&\rho\in\RR,\ \mH\in(\cy)^*,\ \eta\in\RR_+^s,&
	\end{aligned}\right.
\end{equation}
and its Lagrangian dual
\begin{equation}\label{eq::f*rdual}
\left\{
	\begin{aligned}
		r^{\dsdp}:=\inf_{\mathscr{L}}\
		&\mathscr{L}(f)&\\
		\text{s.t.}\
		& \mathscr{L}\in(\cs)^*,\ \mathscr{L}(g)=1,&\\
		& -\mL(p(x,y))\in\cy,\ \mL(\psi_j)\le 0,\
		j=1,\ldots,s.&
	\end{aligned}\right.
\end{equation}

For simplicity, in what follows, we adopt the notation
\[
	\mL(x):=(\mL(x_1),\ldots,\mL(x_m))
\]
for any $\mL\in(\RR[x])^*$.  Let
\begin{equation}\label{eq::d}
	\mathbf{d}:=\max\{\deg(f), \deg(g), \deg(\psi_1), \ldots,
		\deg(\psi_s), \deg_{\bX}(p(\bX,\bY))\}.
\end{equation}
For any $\varepsilon\ge 0$, denote the set of
$\varepsilon$-optimal solutions of \eqref{FMP}
\begin{equation}\label{eq::Se}
	\bs_{\varepsilon}:=\left\{x\in\K\ \Big|\ \frac{f(x)}{g(x)}\le
	r^{\star}+\varepsilon\right\}.
\end{equation}

The following results show that if the cones $\cs$ and $\cy$ satisfy
certain conditions, we can approximate $r^{\star}$ and extract an
$\varepsilon$-optimal solution by solving \eqref{eq::f*r} and
\eqref{eq::f*rdual}.
\begin{theorem}\label{th::main}
	Suppose that {\sf (A1-2)} hold and $\cy\subseteq\mathscr{P}(\Y)$. 
	For some $\varepsilon\ge 0,$ suppose that 
	there exists some $u^{(\varepsilon)}\in\bs_\varepsilon$ such that
	$-p(u^{(\varepsilon)},y)\in\cy$ and $h(u^{(\varepsilon)})\ge 0$
	for any $h(x)\in\cs$.
	\begin{enumerate}[\upshape (i)]
		\item If {\sf (A3)} holds and $L_{f,g}(x,\mu^{\star},\eta^{\star})+\varepsilon
			g(x)\in\cs$, 
			then $r^{\star}-\varepsilon\le r^{\psdp}\le
			r^{\dsdp}\le r^{\star}+\varepsilon$.
\item If $\mathscr{L}^{\star}$ is a minimizer of \eqref{eq::f*rdual}
	such that the restriction $\mL^{\star}|_{\RR[x]_{\mathbf{d}}}$
	admits a representing nonnegative measure $\nu,$ then $r^{\star}\le
	r^{\dsdp}\le r^{\star}+\varepsilon$ and  
	\[
		\frac{\mathscr{L}^{\star}(\bX)}{\mathscr{L}^{\star}(1)}=
		\frac{1}{\int d\nu}\left(\int x_1 d\nu,\ldots,\int x_m
		d\nu\right)\in \bs_{\varepsilon}. 
	\]
	\end{enumerate}
\end{theorem}
\begin{proof}
	Define a linear functional $\mL'\in(\RR[x])^*$ such that
$\mL'(x^\alpha)=\frac{(u^{\varepsilon})^\alpha}{g(u^{\varepsilon})}$ for each $\alpha\in\N^m$.
By the assumption, it is clear that $\mathscr{L}'$ is feasible to
\eqref{eq::f*rdual}. Then, $r^{\dsdp}\le
\mathscr{L}'(f)=\frac{f(u^{\varepsilon})}{g(u^{\varepsilon})}\le
 r^{\star}+\varepsilon$. 

 (i)	Define $\mH^{\star}\in(\RR[y])^*$ by letting
 $\mH^{\star}(y^\beta)=\int_{\Y} y^\beta d\mu^{\star}(\by)$ for any
 $\beta\in\N^n$. By the assumption, $\mH^{\star}\in(\cy)^*$. 
Since $L_{f,g}(x,\mu^{\star},\eta^{\star})+\varepsilon g(x)\in\cs$,
$(r^{\star}-\varepsilon,\mH^{\star},\eta^{\star})$ is feasible to
\eqref{eq::f*r}. Hence, $r^{\psdp}\ge r^{\star}-\varepsilon$.
Then, the weak duality implies the conclusion.

(ii) As $\mL^{\star}|_{\RR[x]_{\mathbf{d}}}$
admits a representing nonnegative measure $\nu$, we have $\mL^{\star}(1)=\int
d\nu>0$; otherwise $\mL^{\star}(g)=0$, a contradiction. 
For every $\by\in \Y$, as $p(\bX,y)$ is convex in $\bX$ and
$\frac{\nu}{\int d\nu}$ is a probability measure, by
Jensen's inequality, we have
\begin{equation}\label{eq::J1}
	p\left(\frac{\mathscr{L}^{\star}(\bX)}{\mathscr{L}^{\star}(1)},\by\right)
	\le
	\frac{1}{\int d\nu}\int p(x,\by)d\nu(x)
	=\frac{1}{\mL^{\star}(1)}\mathscr{L}^{\star}(p(\bX,\by))\le
	0,
\end{equation}
where the last inequality follows from the constraint of the Lagrangian dual problem \eqref{eq::f*rdual}.
For the same reason,
\begin{equation}\label{eq::J2}
	\psi_{j}\left(\frac{\mathscr{L}^{\star}(\bX)}{\mathscr{L}^{\star}(1)}\right)\le
	0,\ \ j=1,\ldots,s,
\end{equation}
which implies that $\mathscr{L}^{\star}(\bX)/\mathscr{L}^{\star}(1)$ is feasible
to \eqref{FMP}. Therefore,
it holds that
\[
	\begin{aligned}
		r^{\star}+\varepsilon&\ge
		\mathscr{L}^{\star}(f)=\frac{\mathscr{L}^{\star}(f)}{\mathscr{L}^{\star}(g)}
		=\frac{\int f(x)d\nu}{\int g(x)d\nu}
		=\frac{\frac{1}{\int d\nu}\int f(x)d\nu}{\frac{1}{\int d\nu}\int g(x)d\nu}\\
		&\ge
		\frac{f\left(\mathscr{L}^{\star}(\bX)/\mathscr{L}^{\star}(1)\right)}
		{g(\mathscr{L}^{\star}(\bX)/\mathscr{L}^{\star}(1))}\ge
		r^{\star}.\\
	\end{aligned}
\]
In particular, the second inequality above can be easily verified
under {\sf (A2)}. Therefore, we have $r^{\star}\le
	r^{\dsdp}\le r^{\star}+\varepsilon$ and  
$\mathscr{L}^{\star}(\bX)/\mathscr{L}^{\star}(1)\in\bs_{\varepsilon}$. 
\end{proof}

Theorem \ref{th::main} opens the possibilities of constructing SDP
relaxations of \eqref{FMP}. In fact, if we can find suitable cones
$\cs$ and $\cy$ with sum-of-squares structures and satisfy the
conditions in Theorem \ref{th::mainOP}, then \eqref{eq::f*r} and
\eqref{eq::f*rdual} can be reduced to SDP problems and become
tractable; see Sections \ref{sec::SDP} and \ref{sec::SDP4} for details.

\begin{remark}{\rm
		We would like to emphasize that the Slater condition
	used in \cite{GS2020} to guarantee {\sf (A3)} and the 
	convergence of the SDP relaxations proposed therein might fail for
	some applications (see Remark \ref{rk::slaterfails} (i)). Thus, we
	need a weaker constraint qualification for {\sf (A3)}.\qed
}\end{remark}

\subsection{A Constraint Qualification for {\sf (A3)}}

Inspired by Jeyakumar and Li~\cite{Jeyakumar2010}, we consider the
following {\itshape semi-infinite characteristic cone constraint
qualification} (SCCCQ).

For a function $h : \RR^m \rightarrow \RR\cup\{-\infty, +\infty\}$,
denote by $h^*$ the conjugate function of $h$, i.e.,
\[
	h^*(\xi)=\sup_{x\in\RR^m} \{\langle \xi, x\rangle -
	h(x)\},
\]
and by $\epi\ h^*$ the epigraph of $h^*$.  Let
\begin{equation}\label{eq::cones}
	\mathcal{C}_1:=\bigcup_{\mu\in\mathcal{M}(\Y)}\epi\left(\int_{\Y}
	p(\cdot,y)d\mu(y)\right)^*\ \text{and}\
	\mathcal{C}_2:=\bigcup_{\eta\in\RR_+^s}\epi\left(\sum_{j=1}^s\eta_j\psi_j\right)^*.
\end{equation}

\begin{definition}\label{def::scccq}
{\rm SCCCQ} is said to be held for $\K$ if $\mathcal{C}_1+\mathcal{C}_2$ is
closed.
\end{definition}

\begin{remark}\label{rk::2cq}{\rm
	Along with Proposition \ref{prop::cq2}, the following example
	shows that the SCCCQ condition is {\itshape weaker} than the
	Slater condition.  Recall that the Slater condition holds
	for $\K$ if there exists $u\in\RR^m$ such that
	$p(u,y)<0$ for all $y\in \Y$ and $\psi_j(u)<0$ for all
	$j=1,\ldots,s.$
	Consider the set $\K=\{x\in\RR\mid yx\le 0,\
	\forall y\in[-1,1]\}$. Clearly, $\K=\{0\}$ and the Slater
	condition fails. As $s=0$, we only need verify that $\mathcal{C}_1$ is
	closed. It suffices to show that
	$\mathcal{C}_1=\{(w,v)\in\RR^2\mid v\ge 0\}$.
	Indeed, fix a $\mu\in\mathcal{M}([-1,1])$ and a point
	$(w,v)\in\epi\left(\int_{\Y}
	p(\cdot,y)d\mu(y)\right)^*\in\mathcal{C}_1$. Then,
	\begin{equation}\label{eq::wv}
		v\ge\sup_{x\in\RR} \left(wx-\int_{[-1,1]}
		xyd\mu(y)\right)
		=\left\{ \begin{array}{cc}
			0,& \ \text{if } w=\int_{[-1,1]} yd\mu(y)\\
			+\infty, & \ \text{if } w\neq\int_{[-1,1]} yd\mu(y).
		\end{array}\right.
	\end{equation}
Conversely, for any $(w,v)\in\RR^2$ with $v\ge 0$, let
\[
	\tilde{\mu}=\left\{
	\begin{aligned}
		|w|\delta_{\{-1\}},\quad\text{if}\ w<0, \\
		w\delta_{\{1\}},\quad\text{if}\ w\ge 0,
\end{aligned}\right.
\]
where $\delta_{\{-1\}}$ and $\delta_{\{1\}}$ are the Dirac measures
at $-1$ and $1$, respectively. Then, $\td{\mu}\in\mathcal{M}([-1,1])$
and $w=\int_{[-1,1]} yd\td{\mu}(y)$ holds. By \eqref{eq::wv}, we have
$(w,v)\in \epi\left(\int_{\Y} p(\cdot,y)d\td{\mu}(y)\right)^*$.\qed
}\end{remark}

For convex semi-infinite programming problems, we claim that the SCCCQ
guarantees the strong duality and the attachment of the solution in
the dual problem, see the next theorem. Due to its own interest, we give a proof in a
general setting in the Appendix~\ref{appendix}. 

\begin{theorem}\label{th::sd} Under {\sf (A1-2)}$,$ {\rm SCCCQ}
	implies {\sf (A3)}.
\end{theorem}
\begin{proof}
	See Theorem \ref{th::cq}. %\qed
\end{proof}

\section{SDP relaxations with asymptotic convergence}\label{sec::SDP}
In this section, based on Theorem \ref{th::main}, we present an SDP
relaxation method for the FSIPP problem \eqref{FMP} with the index set
$\Y$ being of the form 
\[
	\Y=\{y\in\RR^n \mid q_1(y)\ge 0, \ldots, q_{\kappa}(y)\ge 0\},
	\quad\text{where } q_1, \ldots, q_{\kappa}\in\RR[y]. 
\]
The asymptotic convergence and convergence rate of the SDP relaxations
will be studied. 
We also present some discussions on the stop criterion for such SDP
relaxations.

By Theorem \ref{th::main}, if we can choose a suitable subset
$\mathcal{X}\subset\RR^m$ and construct appropriate approximations
$\cs$ and $\cy$ of
$\mathscr{P}(\mathcal{X})$ and $\mathscr{P}(\Y)$, respectively, which
satisfy the conditions in Theorem \ref{th::main} for some $\varepsilon
> 0$, then we can compute the $\varepsilon$-optimal value of
\eqref{FMP} by solving \eqref{eq::f*r} and \eqref{eq::f*rdual}.
Consequently, to construct an asymptotically convergent hierarchy of
SDP relaxations of \eqref{FMP}, we need find two sequences
$\{\csk{k}\}\subset\RR[x]$ and $\{\cyk{k}\}\subset\RR[y]$, which are
approximations of $\mathscr{P}(\mathcal{X})$ and $\mathscr{P}(\Y)$,
respectively, and have sum-of-squares structures. These sequences of
approximations should meet the requirement that for any
$\varepsilon>0$, there exists some $k_{\varepsilon}\in\N$ such that for any
$k\ge k_{\varepsilon}$, the conditions in Theorem \ref{th::main} will
hold if we replace the notation $\cs$ and $\cy$ by $\csk{k}$ and
$\cyk{k}$, respectively.
We may construct such sequences of
approximations by the Positivstellensatz from real algebraic geometry
(recall Putinar's Positivstellensatz introduced in Section
\ref{sec::pre}) where the subscript $k$ indicates the degree of
polynomials in the approximations and the containment relationship
$\csk{k}\subset\csk{k+1}$, $\cyk{k}\subset\cyk{k+1}$ is satisfied. 
In \eqref{eq::f*r} and \eqref{eq::f*rdual}, replace the notation $\cs$
and $\cy$ by $\csk{k}$ and $\cyk{k}$, respectively, and denote the
resulting problems by $(\ref{eq::f*r}k)$ and $(\ref{eq::f*rdual}k)$.
Then, as $k$ increases, a hierarchy of SDP relaxations of \eqref{FMP}
can be constructed. Denote by $r^{\psdp}_k$ and $r^{\dsdp}_k$ the optimal
values of $(\ref{eq::f*r}k)$ and $(\ref{eq::f*rdual}k),$ respectively.
The argument above is formally stated in the following theorem.

\begin{theorem}\label{th::main1}
	Suppose that {\sf (A1-3)} hold and
	$\cyk{k}\subseteq\mathscr{P}(\Y)$
	for all $k\in\N, k\ge\lceil\mathbf{d}/2\rceil$. For any small
	$\varepsilon>0,$ suppose that 
	there exist $k_{\varepsilon}\in\N,
	k_{\varepsilon}\ge\lceil\mathbf{d}/2\rceil$ and some
	$u^{(\varepsilon)}\in\bs_\varepsilon$ such that for all $k\ge
	k_{\varepsilon}$,
	$-p(u^{(\varepsilon)},y)\in\cyk{k}$, 
	$L_{f,g}(x,\mu^{\star},\eta^{\star})+\varepsilon g(x)\in\csk{k}$,
	and $h(u^{(\varepsilon)})\ge 0$ holds for any $h(x)\in\csk{k}$.
	Then$,$
	$\lim_{k\rightarrow\infty}r_k^{\psdp}=\lim_{k\rightarrow\infty}
	r_k^{\dsdp}=r^{\star}$.
\end{theorem}
\begin{proof}
For any small $\varepsilon>0$, by Theorem \ref{th::main} (i), we have 
$r^{\star}-\varepsilon\le r^{\psdp}_k\le r^{\dsdp}_k\le
r^{\star}+\varepsilon$ for any $k\ge k_{\varepsilon}$ and hence the convergence
follows. 
\end{proof}

\subsection{SDP relaxations with asymptotic convergence}
In what follows, we will construct appropriate cones $\{\csk{k}\}$ and
$\{\cyk{k}\}$, which can satisfy conditions in Theorem \ref{th::main1}
and reduce $(\ref{eq::f*r}k)$ and $(\ref{eq::f*rdual}k)$ to SDP problems.  

Fix two numbers $R>0$ and $g^{\star}>0$
such that 
\begin{equation}\label{eq::rg}
	\Vert u^{\star}\Vert <R \quad \text{and}\quad g(u^{\star})>
	g^{\star} \quad \text{for some}\quad u^{\star}\in\bs.
\end{equation}
\begin{remark}{\rm
	Since $\bs\neq\emptyset$ and {\sf (A2)} holds, the above $R$ and
	$g^{\star}$ always exist. Let us discuss how to choose $R$ and
$g^{\star}$ in some circumstances. If $\K$ or $\bs$ is bounded, then
we can
choose sufficiently large $R$ such that $\K\subset[-R, R]^m$ or
$\bs\subset[-R, R]^m$. Now let us consider that how to choose a
sufficiently small $g^{\star}>0$ such that $g(u^{\star})>g^{\star}$
for some $u^{\star}\in\bs$, 
which may be not easy to certify in practice. If $g(x)\equiv 1$, then
clearly, we can let $g^{\star}=1/2$. 
If $g(x)$ is affine, then we can set $g^{\star}$ by solving
$\min_{x\in\K} g(x)$, 
which is an FSIPP problem in which the denominator in the objective is
one.
Suppose that $g(x)$ is not affine and a feasible point
$u'\in\K$ is known. We first solve the FSIPP problem
$f^{\star}:=\min_{x\in\K} f(x)$.
If $f(u')=0$ or $f^{\star}=0$, then by {\sf (A2)},
$r^{\star}=0$; otherwise, we have $f^{\star}>0$ and 
\[
	g(u^{\star})\ge \frac{g(u')}{f(u')}f(u^{\star})\ge
	\frac{g(u')}{f(u')}f^{\star},
\]
for any $u^{\star}\in\bs$. Thus, we can set
$g^{\star}$ to be a positive number less than
$\frac{g(u')}{f(u')}f^{\star}$. \qed
}
\end{remark}

We choose the subset 
\begin{equation}\label{eq::xset}
	\mathcal{X}:=\{x\in\RR^m \mid \Vert x\Vert^2\le R^2,\ g(x)\ge g^{\star}\}. 
\end{equation}
which clearly satisfies the condition
$\mathcal{X}\cap\bs\neq\emptyset$ in Proposition \ref{prop::eq} and let 
\[
	Q:=\{q_1,\ldots,q_{\kappa}\}\subset\RR[y],\quad G:=\{R^2-\Vert x\Vert^2, \
		g(x)-g^{\star}\}\subset\RR[x]. 
\]
For any $k\in\N, k\ge\lceil\mathbf{d}/2\rceil$, let 
\begin{equation}\label{eq::cy}
	\csk{k}=\qm_k(G)\ \ \text{and}\ \ \cyk{k}=\qm_k(Q),
\end{equation}
i.e., the $k$-th quadratic
modules generated by $G$ and $Q$ in $\RR[x]$ and $\RR[y]$,
respectively.  Then, for each $k\ge\lceil\mathbf{d}/2\rceil$,
computing $r^{\psdp}_k$ and $r^{\dsdp}_k$ is reduced to solving a pair
of primal and dual SDP problems. We omit the detail for simplicity.

Consider the assumption:
\begin{quote}
	{\sf (A4):}
$\qm(Q)$ is Archimedean and there
	exists a point $\bar{u}\in\K$ such that $p(\bar{u},y)<0$
	for all $y\in\Y$.
\end{quote}

\begin{theorem}\label{th::asy}
	Under {\sf (A1-4)} and the settings \eqref{eq::rg} and
	\eqref{eq::cy}$,$ the following holds.
	\begin{enumerate}[\upshape (i) ]
		\item
			$\lim_{k\rightarrow\infty}r_k^{\psdp}=\lim_{k\rightarrow\infty}
			r_k^{\dsdp}=r^{\star}$.
		\item If $r_k^{\dsdp}<+\infty$ which holds for $k$ large
			enough$,$ then $r_k^{\psdp}=r_k^{\dsdp}$ and $r_k^{\dsdp}$
			is attainable.
		\item For any convergent subsequence
			$\{\mL^{\star}_{k_i}(x)/\mL^{\star}_{k_i}(1)\}_i$
			$($always exists$)$ of
			$\{\mL^{\star}_k(x)/\mL^{\star}_k(1)\}_k$  where $\mL_k^{\star}$ is a minimizer of
			$(\ref{eq::f*rdual}k),$ we have
			$\lim_{i\rightarrow\infty}\mL^{\star}_{k_i}(x)/\mL^{\star}_{k_i}(1)\in\bs$.
			Consequently$,$ if $\bs$ is singleton$,$ then 
			$\lim_{k\rightarrow\infty}\mL^{\star}_k(x)/\mL^{\star}_k(1)$
			is the unique minimizer of \eqref{FMP}. 
	\end{enumerate}
\end{theorem}
\begin{proof}
	(i) Clearly, $\cyk{k}\subset\mathscr{P}(\Y)$ for any $k\in\N,
k\ge\lceil\mathbf{d}/2\rceil$. Let
$\varepsilon>0$ be fixed. Let $u^{\star}\in\bs$ be as in
\eqref{eq::rg} and $u^{(\lambda)}:=\lambda
u^{\star}+(1-\lambda)\bar{u}$. As $\K$ is
convex, $u^{(\lambda)}\in\K$ for any $0\le \lambda\le 1$. By the
continuity of $g$ and $\frac{f}{g}$ on $\K$, there exists a
$\lambda'\in(0,1)$
such that 
\begin{equation}\label{eq::rg2}
	\Vert u^{(\lambda')}\Vert<R,\quad
	g(u^{(\lambda')})>g^{\star}\quad\text{and}\quad
	u^{(\lambda')}\in\bs_{\varepsilon}. 
\end{equation}
For any $y\in\Y$, 
by the convexity of $p(x,y)$ in $x$,
\[
	p(u^{(\lambda')},y)\le \lambda' p(u^{\star},y) + (1-\lambda')
	p(\bar{u},y)<0.
\]
By Theorem \ref{th::PP}, there exists a $k_1\in\N$ such that
$-p(u^{(\lambda')},y)\in \cyk{k}$ for any $k\ge k_1$. 
Since $g^{\star}>0$, {\sf (A3)} implies that 
$L_{f,g}(x,\mu^{\star},\eta^{\star})+\varepsilon g(x)$ is positive on
the set $\mathcal{X}$. 
As $\qm(G)$ is Archimedean, by Theorem \ref{th::PP} again, there
exists a $k_2\in\N$ such that
$L_{f,g}(x,\mu^{\star},\eta^{\star})+\varepsilon g(x)\in\csk{k}$ for
any $k\ge k_2$. 
It is obvious from \eqref{eq::rg2} that $h(u^{(\lambda')})\ge 0$ for
any $h\in\csk{k}$, $k\ge k_2$.
Let $k_{\varepsilon}=\max\{k_1, k_2, \lceil\mathbf{d}/2\rceil\}$, then the sequences
$\{\csk{k}\}$ and $\{\cyk{k}\}$ satisfies the conditions in Theorem
\ref{th::main1} which implies the conclusion.

(ii) From the above, the linear functional $\mL'\in(\RR[x])^*$ such that
$\mL'(x^\alpha)=\frac{(u^{\lambda'})^\alpha}{g(u^{\lambda'})}$ for
each $\alpha\in\N^m$ is feasible for $(\ref{eq::f*rdual}k)$ whenever
$k\ge\max\{k_1,\lceil\mathbf{d}/2\rceil\}$ and hence $r^{\dsdp}_k<
+\infty$. For any $k\ge
\max\{k_1,\lceil\mathbf{d}/2\rceil\}$ and any feasible point $\mL_k$ of
$(\ref{eq::f*rdual}k)$, because
$\mL_k\in(\csk{k})^*$, we have $\mL_k(g-g^{\star})\ge 0$ and
$\mL_k(1)\ge 0$. Hence, along with $\mL_k(g)=1$, we have $0\le
\mL_k(1)\le 1/g^{\star}$ for all $k\ge \lceil\mathbf{d}/2\rceil$. 
Since there is a ball constraint in \eqref{eq::xset}, by \cite[Lemma
3]{CD2016} and its proof, we have
\[
	\sqrt{\sum_{\alpha\in\N^m_{2k}}\left(\mL_k(x^{\alpha})\right)^2}
	\le\mL_k(1) \sqrt{\binom{m+k}{m}}\sum_{i=0}^k R^{2i}
	\le \frac{1}{g^{\star}} \sqrt{\binom{m+k}{m}}\sum_{i=0}^k R^{2i}
\]
for all $k\in \N$, $k\ge \lceil\mathbf{d}/2\rceil$ and all 
$\mL_k\in(\csk{k})^*$. In other words, for any $k\ge
\max\{k_1,\lceil\mathbf{d}/2\rceil\}$, the feasible set of the
$(\ref{eq::f*rdual}k)$ is nonempty, bounded and closed. Then,
the solution set of the
$(\ref{eq::f*rdual}k)$ is nonempty and bounded, which implies that
$(\ref{eq::f*r}k)$ is strictly feasible (c.f. \cite[Section
4.1.2]{SS2000}). Consequently, the strong duality
$r_k^{\psdp}=r_k^{\dsdp}$ holds by \cite[Theorem 4.1.3]{SS2000}.

(iii) As $\qm(G)$ is Archimedean, by the definition,  
\[
	\forall t\in\N,\ \exists N_t,\ l(t)\in\N,\ \forall \alpha\in\N_t^m,\
	N_t\pm x^{\alpha}\in\qm_{l(t)}(G)=\csk{l(t)}. 
\]
For any $k\ge l(t)$, since $\mL^{\star}_k\in(\csk{k})^*$, for all
$\alpha\in\N_t^m$, 
\begin{equation}\label{eq::B}
	|\mL^{\star}_k(x^{\alpha})|\le
	\mL^{\star}_k(N_t)=N_t\cdot\mL^{\star}_k(1)\le N_t/g^{\star}. 
\end{equation}
Then, for any $k\ge\lceil \mathbf{d}/2\rceil$, we
have $|\mL^{\star}_k(x^{\alpha})|\le N_t'$ for any $\alpha\in\N_t^m$
where 
\[
	N_t':=\max\{N_t/g^{\star}, M_t\}\quad\text{and}\quad	M_t:=\max\{|\mL^{\star}_k(x^{\alpha})| \mid \alpha\in\N_t^m,
	\lceil\mathbf{d}/2\rceil\le k\le l(t)\}.
\]
Moreover, from \eqref{eq::B} and the equality $\mL^{\star}_k(g)=1$, 
we can see that $\mL^{\star}_k(1)>0$ for all
$k\ge\lceil\mathbf{d}/2\rceil$.
For any $k\ge \lceil\mathbf{d}/2\rceil$, extend
$\mL^{\star}_k\in(\RR_{2k}[x])^*$ to $(\RR[x])^*$ by letting
$\mL^{\star}_k(x^{\alpha})=0$ for all $|\alpha|>2k$ and denote it by
$\wt{\mL}^{\star}_k$. Then, for any $k\ge\lceil\mathbf{d}/2\rceil$ and
any $\alpha\in\N^m$, it holds that $|\wt{\mL}^{\star}_k(x^{\alpha})|\le
N_{|\alpha|}'$. That is, 
\begin{equation}\label{eq::product}
	\left\{\left(\wt{\mL}^{\star}_k(x^{\alpha})\right)_{\alpha\in\N^m}
\right\}_{k\ge\lceil\mathbf{d}/2\rceil}\subset
\prod_{\alpha\in\N^m}\left[-N_{|\alpha|}', N_{|\alpha|}'\right]. 
\end{equation}
By Tychonoff's theorem, the product space on the right side of
\eqref{eq::product} is compact in the product topology. 
Therefore, there exists a 
subsequence $\{\wt{\mL}^{\star}_{k_i}\}_{i\in\N}$ of
$\{\wt{\mL}^{\star}_k\}_k$ and a $\wt{\mL}^{\star}\in(\RR[x])^*$ such that
$\lim_{i\rightarrow\infty}\wt{\mL}^{\star}_{k_i}(x^{\alpha})=\wt{\mL}^{\star}(x^{\alpha})$
for all $\alpha\in\N^m$. From the pointwise convergence, we get the
following: (a) $\wt{\mL}^{\star}\in(\qm(G))^*$;
(b) $\wt{\mL}^{\star}(g)=1$; (c) $\wt{\mL}^{\star}(p(x,y))\le 0$ for
any $y\in\Y$ since $\wt{\mL}^{\star}_k(p(x,y))\le 0$ for any
$k\ge\lceil\mathbf{d}/2\rceil$; (d) $\wt{\mL}^{\star}(\psi_j)\le
0$ for $j=1,\ldots,s$. By (a) and Putinar's Positivstellensatz, along
with Haviland's theorem \cite{Haviland1936}, 
$\wt{\mL}^{\star}$ admits a representing nonnegative measure $\nu$, i.e.,
$\wt{\mL}^{\star}(x^{\alpha})=\int x^{\alpha} d\nu$ for all
$\alpha\in\N^m$. From (b) and \eqref{eq::B}, $\wt{\mL}^{\star}(1)>0$.
Then, like in \eqref{eq::J1} and \eqref{eq::J2}, by (c) and (d), we
can see that
\[
	\lim_{i\rightarrow\infty}\frac{\mL^{\star}_{k_i}(x)}{\mL^{\star}_{k_i}(1)}
	=\frac{\wt{\mL}^{\star}(x)}{\wt{\mL}^{\star}(1)}\in\K.
\]
From (i), 
\[
	\begin{aligned}
		r^{\star}=\lim_{i\rightarrow\infty}\wt{\mL}^{\star}_{k_i}(f)&=
		\wt{\mL}^{\star}(f)=\frac{\wt{\mL}^{\star}(f)}{\wt{\mL}^{\star}(g)}
		=\frac{\int f(x)d\nu}{\int g(x)d\nu}
		=\frac{\frac{1}{\int d\nu}\int f(x)d\nu}{\frac{1}{\int d\nu}\int g(x)d\nu}\\
		&\ge
		\frac{f\left(\wt{\mL}^{\star}(\bX)/\wt{\mL}^{\star}(1)\right)}
		{g(\wt{\mL}^{\star}(\bX)/\wt{\mL}^{\star}(1))}\ge r^{\star},\\
	\end{aligned}
\]
which implies that
$\lim_{i\rightarrow\infty}\mL^{\star}_{k_i}(x)/\mL^{\star}_{k_i}(1)\in\bs$.

As $\bs$ is singleton, $\bs=\{u^{\star}\}$. 
The above arguments show that $\lim_{i\rightarrow\infty}
\mL^{\star}_{k_i}(x)/\mL^{\star}_{k_i}(1)=u^{\star}$ for any
convergent subsequence of $\{\mL^{\star}_k(x)/\mL^{\star}_k(1)\}_k$. 
By \eqref{eq::product}, 
$\{\mL^{\star}_k(x)/\mL^{\star}_k(1)\}_k\subset[-N'_1, N'_1]^m$ which is
bounded.  Thus, the whole sequence
$\{\mL^{\star}_k(x)/\mL^{\star}_k(1)\}_k$ converges to $u^{\star}$ as
$k$ tends to $\infty$. 
\end{proof}

\subsection{Convergence rate analysis}
Next, we give some convergence rate analysis of $ r_k^{\psdp}$ and
$r_k^{\dsdp}$ based on Theorem \ref{th::complexity}.

Let us fix $R, g^{\star}\in\RR$, $u^{\star}\in\K$ satisfying \eqref{eq::rg}, 
$\mu^{\star}\in\mathcal{M}(\Y), \eta^{\star}\in\RR^s_+$ satisfying
{\sf (A3)}, $\bar{u}\in\K$ satisfying {\sf (A4)}, a number
$R_{\mathcal{X}}>R$ and 
a number $R_{\Y}>0$ such that $\Y\subset (-R_{\Y}, R_{\Y})^n$.

For any $\varepsilon>0$, define the following constants.
\[
	\begin{array}{ll}
		N_e:=\left\{
			\begin{array}{ll}
				\frac{\Vert \bar{u}\Vert-R}{\Vert\bar{u}\Vert-\Vert u^{\star}\Vert},&
					\text{if}\ \Vert \bar{u}\Vert\ge R, \\
					0,& \text{otherwise},
				\end{array}
				\right.&
				N_g:=\left\{
					\begin{array}{ll}
						\frac{g(\bar{u})-g^{\star}}{g(\bar{u})-g(u^{\star})},&
						\text{if}\ g(\bar{u})\le g^{\star}, \\
						0,& \text{otherwise},
					\end{array}
					\right.\\
		N_f:=\left\{
			\begin{array}{ll}
				\frac{f(\bar{u})-(r^{\star}+\varepsilon)g(\bar{u})}
				{f(\bar{u})-(r^{\star}+\varepsilon)g(\bar{u})+\varepsilon
				g(u^{\star})},&
				\text{if}\ f(\bar{u})\ge
				(r^{\star}+\varepsilon)g(\bar{u}), \\
					0,& \text{otherwise},
				\end{array}
				\right. &
			N_{\varepsilon}:=\max\{N_e, N_g, N_f\}.
	\end{array}
\]
It is easy to see that $N_{\varepsilon}\in[0, 1)$.
Let 
\begin{equation}\label{eq::ul}
	\lambda':=\frac{N_{\varepsilon}+1}{2}\quad\text{and}\quad 
	u^{(\lambda')}=\lambda' u^{\star}+(1-\lambda')\bar{u}. 
\end{equation}
\begin{lemma}\label{lem::N}
	The point $u^{(\lambda')}$ in \eqref{eq::ul} satisfies the
	conditions in \eqref{eq::rg2}. 
\end{lemma}
\begin{proof}
	If $\Vert\bar{u}\Vert<R$, then clearly $\Vert u^{(\lambda')}\Vert<R$;
otherwise, $\Vert \bar{u}\Vert\ge R>\Vert u^{\star}\Vert$ and 
\[
	\begin{aligned}
		\Vert u^{(\lambda')}\Vert=
		\Vert \lambda'u^{\star}+(1-\lambda')\bar{u}\Vert&\le \lambda'\Vert u^{\star}\Vert
		+(1-\lambda')\Vert\bar{u}\Vert\\
		&<\Vert \bar{u}\Vert +N_e(\Vert
		u^{\star}\Vert-\Vert\bar{u}\Vert)= \Vert
		u^{\star}\Vert+R-\Vert u^{\star}\Vert=R.
	\end{aligned}
\]
Since $g(x)$ is concave, we have
\[
	g(u^{(\lambda')})\ge \lambda' g(u^{\star})+(1-\lambda')
	g(\bar{u}). 
\]
If $g(\bar{u})>g^{\star}$, it is clear that
$g(u^{(\lambda')})>g^{\star}$. Suppose that $g(\bar{u})\le
g^{\star}$, then $g(\bar{u})<g(u^{\star})$ and 
\[
	\begin{aligned}
		g(u^{(\lambda')})&\ge \lambda' g(u^{\star})+(1-\lambda')
		g(\bar{u})=\lambda'(g(u^{\star})-g(\bar{u}))+g(\bar{u})\\
		&>N_g(g(u^{\star})-g(\bar{u}))+g(\bar{u})=g^{\star}-g(\bar{u})+g(\bar{u})=g^{\star}.
	\end{aligned}
\]
If $f(\bar{u})<(r^{\star}+\varepsilon)g(\bar{u})$, by the convexity of
$f(x)$ and $-g(x)$, it holds that 
\begin{equation}\label{eq::ineq}
	f(u^{(\lambda')})\le \lambda' f(u^{\star})+(1-\lambda')
	f(\bar{u})<
	(r^{\star}+\varepsilon)(\lambda' g(u^{\star})+(1-\lambda')
	g(\bar{u}))
	\le (r^{\star}+\varepsilon) g(u^{(\lambda')}), 
\end{equation}
which implies that $u^{(\lambda')}\in\bs_{\varepsilon}$. If $f(\bar{u})\ge
(r^{\star}+\varepsilon)g(\bar{u})$, we have 
\[
	\begin{aligned}
		&\lambda'(f(\bar{u})-(r^{\star}+\varepsilon)g(\bar{u}))-\lambda'(f(u^{\star})-(r^{\star}+\varepsilon)g(u^{\star}))\\
		&> N_f((f(\bar{u})-(r^{\star}+\varepsilon)g(\bar{u}))+\varepsilon
		g^{\star})\\
		&=f(\bar{u})-(r^{\star}+\varepsilon)g(\bar{u}).
	\end{aligned}
\]
Then, the second inequality of \eqref{eq::ineq} still holds and hence
$u^{(\lambda')}\in\bs_{\varepsilon}$.
Therefore, all conditions in
\eqref{eq::rg2} are satisfied by $u^{(\lambda')}$.
\end{proof}
Recall the norm defined in \eqref{eq::fnorm}.  Write 
$p(x,y)=\sum_{\beta\in\N^n}p_{y,\beta}(x) y^{\beta}$ and let 
\[
	p_{\max}:=\max_{\beta\in\N^n}\frac{\max_{\Vert x\Vert\le R}
	\vert p_{y,\beta}(x)\vert}{\binom{|\beta|}{\beta}}.
\]
Then, $p_{\max}$ is well-defined and $\Vert
p(u^{(\lambda')},y)\Vert\le p_{\max}$ by Lemma \ref{lem::N}. 
Denote $p^{\star}_{\bar{u}}:=\max_{y\in\Y} p(\bar{u},y)$. As $\Y$ is
compact, $p^{\star}_{\bar{u}}<0$. 
Denote $d_y=\deg_y p(x,y)$ and $L_{\max}:= \Vert
L_{f,g}(x,\mu^{\star},\eta^{\star})\Vert$ for simplicity. 
The convergence rate analysis of $ r_k^{\psdp}$ and $r_k^{\dsdp}$ is
presented in Proposition \ref{prop::rate}, where the only constant
depending on $\varepsilon$ is $N_\varepsilon$, and all
others depend on the problem data in \eqref{FMP} and the fixed $R,
g^{\star}, u^{\star}, \bar{u}, \mu^{\star}$, $\eta^{\star}$,
$R_{\mathcal{X}}$ and $R_{\Y}$ in the assumptions.

\begin{proposition}\label{prop::rate}
	Under {\sf (A1-4)} and the settings \eqref{eq::rg} and
	\eqref{eq::cy}$,$  
	there exist constants $c_1,\ c_2\in\RR$ $($depending on $q_i$'s$,$ $g,$
	$R$ and $g^{\star}$$)$ such that for any
	$\varepsilon>0,$ we have $r^{\star}-\varepsilon\le r_k^{\psdp}\le
	r^{\dsdp}_k\le r^{\star}+\varepsilon$ whenever
	\[
		k\ge\max\left\{
		c_1\exp\left[\left(d_y^2n^{d_y}
		\frac{2p_{\max}R_{\Y}^{d_y}}{(N_\varepsilon-1)p^{\star}_{\bar{u}}}\right)^{c_1}\right], 
	c_2\exp\left[\left(\mathbf{d}^2m^{\mathbf{d}}
		\frac{(L_{\max}+\varepsilon \Vert g(x)\Vert)R_{\mathcal{X}}^{\mathbf{d}}}{\varepsilon
		g^{\star}}\right)^{c_2}\right],
		\lceil\mathbf{d}/2\rceil\right\}. 
	\]
\end{proposition}
\begin{proof}
	Recall the proof of Theorem \ref{th::asy} (i). By Lemma
	\ref{lem::N}, $u^{(\lambda')}$ in \eqref{eq::ul} satisfies the
	conditions in \eqref{eq::rg2}. Then, 
by Theorem \ref{th::complexity}, there is a constant $c_1\in\RR$ depending
only on $q_i$'s such that $-p(u^{(\lambda')},y)\in\cyk{k}$ for all
$k\ge k_1$ where
\[
k_1:=c_1\exp\left[\left(d_y^2n^{d_y}
\frac{\Vert p(u^{(\lambda')},y)\Vert R_{\Y}^{d_y}}
{\min_{y\in\Y}(-p(u^{(\lambda')},y))}\right)^{c_1}\right],
\]
and there exists a constant $c_2\in\RR$ depending only on $g(x), R$
and $g^{\star}$ such that
$L_{f,g}(x,\mu^{\star},\eta^{\star})+\varepsilon g(x)\in\csk{k}$ for
all $k\ge k_2$ where
\[
k_2:= c_2\exp\left[\left(\mathbf{d}^2m^{\mathbf{d}}
	\frac{\Vert L_{f,g}(x,\mu^{\star},\eta^{\star})
	+\varepsilon g(x)\Vert R_{\mathcal{X}}^{\mathbf{d}}}
{\min_{x\in\mathcal{X}}(L_{f,g}(x,\mu^{\star},\eta^{\star})
+\varepsilon g(x))}\right)^{c_2}\right]. 
\]
For any $y\in\Y$, by the convexity of $p(x,y)$ in $x$,
\[
	-p(u^{(\lambda')},y)\ge -\lambda' p(u^{\star},y) - (1-\lambda')
	p(\bar{u},y)\ge(\lambda'-1)p_{\bar{u}}^{\star}
	=\frac{(N_\varepsilon-1)p_{\bar{u}}^{\star}}{2}>0. 
\]
Moreover, $L_{f,g}(x,\mu^{\star},\eta^{\star})+\varepsilon g(x)\ge
\varepsilon g^{\star}$
on the set $\mathcal{X}$ in \eqref{eq::xset}. Therefore, 
\[
k_1\le c_1\exp\left[\left(d_y^2n^{d_y}
\frac{2p_{\max}R_{\Y}^{d_y}}{(N_\varepsilon-1)p^{\star}_{\bar{u}}}\right)^{c_1}\right]\quad
\text{and}\quad
k_2\le 
c_2\exp\left[\left(\mathbf{d}^2m^{\mathbf{d}}
	\frac{(L_{\max}+\varepsilon \Vert g(x)\Vert) R_{\mathcal{X}}^{\mathbf{d}}}
{\varepsilon g^{\star}}\right)^{c_2}\right]. 
\]
Then, the conclusion follows. 
\end{proof}

\subsection{Discussions on the stop criterion}
Recall the asymptotic convergence of the hierarchy of SDP
relaxations $(\ref{eq::f*r}k)$ and $(\ref{eq::f*rdual}k)$ for the
FSIPP problem \eqref{FMP} established in Theorem \ref{th::asy}.  
Before we give an example to show the efficiency of this method, let
us discuss how to check whether or not
$\mL_k^{\star}(x)/\mL_k^{\star}(1)$  where $\mL^{\star}_k$ is a
minimizer of $(\ref{eq::f*rdual}k)$ for some $k$ is a satisfying
solution to \eqref{FMP}. 

Under {\sf (A1-2)}, it is clear that a feasible point $u^{\star}\in\K$
is a minimizer of \eqref{FMP} if and only if $u^{\star}$ is a 
minimizer of the convex semi-infinite programming problem 
\begin{equation}\label{eq::csipu}
	\min_{x\in\K}\ f(x)-\frac{f(u^{\star})}{g(u^{\star})} g(x). 
\end{equation}
For \eqref{eq::csipu}, it is well-known \cite{Still2007} that if the
KKT condition holds at $u^{\star}\in\K$, i.e. there are finite subsets
$\Lambda(u^{\star})\subset\Y$, $J(u^{\star})\subset\{1,\ldots,s\}$ and multipliers
$\gamma_y\ge 0$, $y\in\Lambda(u^{\star})$, $\eta_j\ge 0$, $j\in
J(u^{\star})$ such that
\begin{equation}\label{eq::KKT}
	\begin{aligned}
		& p(u^{\star},y)=0,\ \forall y\in\Lambda(u^{\star}),\
		\psi_j(u^{\star})=0,\ \forall j\in
		J(u^{\star}), \\
		& \nabla f(u^{\star})-\frac{f(u^{\star})}{g(u^{\star})} \nabla
		g(u^{\star}) +
		\sum_{y\in\Lambda(u^{\star})} \gamma_{y} \nabla_x p(u^{\star},y) + \sum_{j\in
			J(u^{\star})} \eta_j \psi_j(u^{\star})=0, 
	\end{aligned}
\end{equation}
then $u^{\star}$ is a minimizer of \eqref{eq::csipu}. The converse
holds if $\K$ satisfies the Slater condition. Next, we
use this fact to give a stop criterion of the hierarchy of SDP
relaxations $(\ref{eq::f*rdual}k)$ for \eqref{FMP}.

Recall Lasserre's SDP relaxation method for polynomial optimization
problems introduced in Section \ref{sec::pre}. Fix a $k\in\N$ and let
$u^{\star}=\mL_k^{\star}(x)/\mL_k^{\star}(1)$. Denote by $\tau$ a small positive
number as a given tolerance. Now, we proceed with the following steps:
\begin{enumerate}[Step 1.]
	\item Solve the polynomial minimization problem 
		\begin{equation}\label{eq::pu}
			p^{\star}:=\min_{y\in\Y} -p(u^{\star}, y)
		\end{equation}
		by Lasserre's SDP relaxation method
		\eqref{eq::lppp} using the software {\sf GloptiPoly}. If
		\[
		\max\{-p^{\star}, \psi_1(u^{\star}),\ldots,
		\psi_s(u^{\star})\}\le \tau,
		\]
		then $u^{\star}$ is a feasible
		point of \eqref{FMP} within the tolerance $\tau$. 
		In the case when Condition \ref{con::truncation} holds in Lasserre's
		relaxations, we can extract the set of
		global minimizers of \eqref{eq::pu} which is a finite set in
		this case (c.f.  \cite{CFKmoment,LasserreHenrion}) and we denote it by
		$\Lambda(u^{\star})$. 
		Let $J(u^{\star}):=\{j\mid
			|\psi_j(u^{\star})|\le \tau\}$, then
		$\Lambda(u^{\star})\cup J(u^{\star})$ indexes the active
		constraints at $u^{\star}$ within the tolerance $\tau$.
	\item Solve the non-negative least-squares problem
		\begin{equation}\label{eq::lsq}
				\qquad\omega:=\min_{\gamma_y\ge 0, \eta_j\ge 0} \Big\Vert \nabla
				f(u^{\star})-\frac{f(u^{\star})}{g(u^{\star})} \nabla
				g(u^{\star}) + \sum_{y\in\Lambda(u^{\star})}
				\gamma_{y} \nabla_x p(u^{\star},y) + \sum_{j\in
					J(u^{\star})} \eta_j \psi_j(u^{\star})\Big\Vert^2,
				\end{equation}
			which can be done by the command {\itshape lsqnonneg} in
			Matlab. If $\omega\le\tau$, then the KKT condition in
			\eqref{eq::KKT} holds at
			$\mL_k^{\star}(x)/\mL_k^{\star}(1)$ within the 
			tolerance $\tau$. Then we may terminate the SDP
			relaxations $(\ref{eq::f*rdual}k)$ at the order $k$ and output
			$\mL_k^{\star}(x)/\mL_k^{\star}(1)$ as a numerical
			minimizer of \eqref{FMP}.
\end{enumerate}
The key of the above procedure is Condition \ref{con::truncation}
which can certify the finite convergence of Lasserre's relaxations for
\eqref{eq::pu} and be used to extract the set $\Lambda(u^{\star})$. 
For a polynomial minimization problem with {\itshape generic}
coefficients data, Nie proved that Condition \ref{con::truncation}
holds in its Lasserre's SDP relaxations (c.f. \cite[Theorem
1.2]{NieFiniteCon} and \cite[Theorem 2.2]{NieFlatTruncation}). 
Hence, an interesting problem is that
if the coefficients data in \eqref{FMP} is generic, does Condition
\ref{con::truncation} always hold in Lasserre's SDP relaxations of
\eqref{eq::pu}? It is not clear to us yet because the
coefficients of $p(u^{\star},y)$ also depend on the solutions
$\mL^{\star}_k$ to $(\ref{eq::f*rdual}k)$ and thus we leave it for
future research. 

\vskip 10pt

Several numerical examples will be presented in the rest of this paper
to show the efficiency of the corresponding SDP relaxations. We use
the software {\sf Yalmip} \cite{YALMIP} and call the SDP solver SeDuMi
\cite{Sturm99} to implement and solve the resulting SDP problems
\eqref{eq::f*r} and \eqref{eq::f*rdual}. To show the advantage of our
SDP relaxation method for solving FSIPP problems, we compare it with
the numerical method called {\itshape adaptive convexification
algorithm}
\footnote{Its code named SIPSOLVER is available at
\url{https://kop.ior.kit.edu/791.php}} (ACA for short)
\cite{ACA2007,ACA2012} for the following reasons. On the one hand,
if $g(x)$ is not a constant function, then the FSIPP problem
\eqref{FMP} is usually not convex.
Hence,  
numerical methods in the literature for convex SIP problems
\cite{Goberna2017,Hettich1993,Still2007} are not
appropriate for \eqref{FMP}. On the other hand, most of the
existing numerical methods for SIP require the index set to be box-shaped, while the
ACA method can solve SIP problems with arbitrary, not necessarily
box-shaped, index sets (as $\Y$ in \eqref{FMP} is). 
The ACA method can deal with general SIP problems (the involved
functions are not necessarily polynomials) by two procedures.
The first phase is to find a consistent initial
approximation of the SIP problem with a reduced outer approximation of
the index set. The second phase is to compute an {\itshape
	$\varepsilon$-stationary point} of the SIP problem by adaptively
	constructing convex relaxations of the lower level problems.
All numerical experiments in the sequel were carried out on a PC with  
two 64-bit Intel Core i5 1.3 GHz CPUs and 8G RAM.

\begin{example}\label{ex::sym}{\rm
In order to construct an illustrating example which is not in the
special cases studied in the next section, we consider the following
two convex but not s.o.s-convex polynomials where $h_1$ is given in
\cite[(4)]{Ahmadi2012} and $h_2$ is given in \cite[(5.2)]{APgap}
\begin{equation}\label{eq::nonsosconvex}
	\begin{aligned}
		h_1(x_1,x_2,x_3)=&32x_1^8+118x_1^6x_2^2+40x_1^6x_3^2+25x_1^4x_2^4
		-43x_1^4x_2^2x_3^2-35x_1^4x_3^4+3x_1^2x_2^4x_3^2\\
		&-16x_1^2x_2^2x_3^4+24x_1^2x_3^6+16x_2^8
		+44x_2^6x_3^2+70x_2^4x_3^4+60x_2^2x_3^6+30x_3^8.\\
		h_2(x_1,x_2)=&89-363x_1^4x_2+\frac{51531}{64}x_2^6-\frac{9005}{4}x_2^5+
	\frac{49171}{16}x_2^4+721x_1^2-2060x_2^3\\
	&-14x_1^3+\frac{3817}{4}x_2^2+363x_1^4-9x_1^5+77x_1^6+316x_1x_2+49x_1x_2^3\\
	&-2550x_1^2x_2-968x_1x_2^2+1710x_1x_2^4+794x_1^3x_2+\frac{7269}{2}x_1^2x_2^2\\
	&-\frac{301}{2}x_1^5x_2+\frac{2143}{4}x_1^4x_2^2+\frac{1671}{2}x_1^3x_2^3+\frac{14901}{16}x_1^2x_2^4-\frac{1399}{2}x_1x_2^5\\
	&-\frac{3825}{2}x_1^3x_2^2-\frac{4041}{2}x_1^2x_2^3-364x_2+48x_1.
				\end{aligned}
			\end{equation}
It can be verified by {\sf Yalmip} that both $h_1$ and $h_2$ are s.o.s polynomials. 
Let
\[
	p(x_1,x_2,y_1,y_2):=-1+h_1(y_1x_1-y_2x_2,y_2x_1+y_1x_2,1)/100+(y_1x_1-y_2x_2)-(y_2x_1+y_1x_2)
\]
and $f(x_1,x_2):=h_2(x_1-1,x_2-1)/10000$. Clearly, $f(x)$ and
$p(x,y)$ for all $y\in\RR^2$ are convex but not s.o.s-convex in $x$. 
Let $g(x_1,x_2):=-x_1^2-x_2^2+4$ and
$\psi(x_1,x_2):=x_1^2/2+2x_2^2-1$. 

Consider the FSIPP problem
\begin{equation}\label{eq::ex1}
		r^{\star}:=\min_{x\in\RR^2}\  \frac{f(x)}{g(x)}\quad
		\text{s.t.}\ \psi(x)\le 0, \ p(\bx,\by)\le 0,\ \ \forall \by\in \Y\subset\RR^2,
\end{equation}
where
\[
	\Y:=\{(y_1,y_2)\in\RR^2 \mid y_1\ge 0,\ y_2\ge 0,\
	y_1^2+y_2^2=1\}. 
\]
Geometrically, the feasible region $\K$ is constructed in the
following way:
first rotate the shape in the $(x_1,x_2)$-plane defined by 
$-1+h_1(x_1,x_2,1)/100-x_1+x_2\le 0$ continuously around the
origin by $90^\circ$ clockwise; then intersect the common area of
these shapes in this process with the ellipse defined by $\psi(x)\le
0$ (see Figure \ref{fig::exsym}). It is easy to see that {\sf (A1-4)}
hold for this problem. Let $R=2$ and $g^{\star}=1$. 
For the first order $k=4$, we get $r_4^{\dsdp}=0.0274$ and
$\mL_4^{\star}(x)/\mL_4^{\star}(1)=(0.7377, 0.6033)$.

As we have discussed before this example, now let us check that if
$u^{\star}:=(0.7377, 0.6033)$ is a satisfying solution to 
\eqref{eq::ex1} within the tolerance $\tau=10^{-3}$. We first solve
the problem \eqref{eq::pu} by Lasserre's SDP relaxations in  {\sf
GloptiPoly}. It turns out that Condition \ref{con::truncation} is
satisfied in Lasserre's relaxations of the first order. We obtain that
$p^{\star}=6.7654\times 10^{-5}$ and $\Lambda(u^{\star})=\{(0.775,
0.6315)\}$. Since $\psi(u^{\star})=4.2425\times 10^{-5}$, within the
tolerance $10^{-3}$, we can see that $u^{\star}$ is a feasible point
of \eqref{eq::ex1} and the constraints
\[
	p(x_1,x_2,0.775,0.6315)\le 0,\quad \psi(x_1,x_2)\le 0,
\]
are active at $u^{\star}$. Then, we solve the non-negative
least-squares problem \eqref{eq::lsq} by the command {\itshape lsqnonneg}
in Matlab. The result is $\omega=0.0000$, which shows that the KKT
condition \eqref{eq::KKT} holds at $u^{\star}$. Thus, $u^{\star}$ is a numerical
minimizer of \eqref{eq::csipu} and hence of \eqref{eq::ex1} within the
tolerance $10^{-3}$. The total
CPU time for the whole process is about 25 seconds.

To show the accuracy of
the solution, we draw some contour curves of $f/g$, including
the one where $f/g$ is the constant value
$f(u^{\star})/g(u^{\star})=0.0274$ (the blue curve), and mark the
point $u^{\star}$ by a red dot in Figure \ref{fig::exsym}. As we
can see, the blue curve is almost tangent to $\K$ at the point
$u^{\star}$, which illustrates the accuracy of $u^{\star}$.
\begin{figure}
	\centering
	\scalebox{0.5}{
		\includegraphics[trim=80 200 80 200,clip]{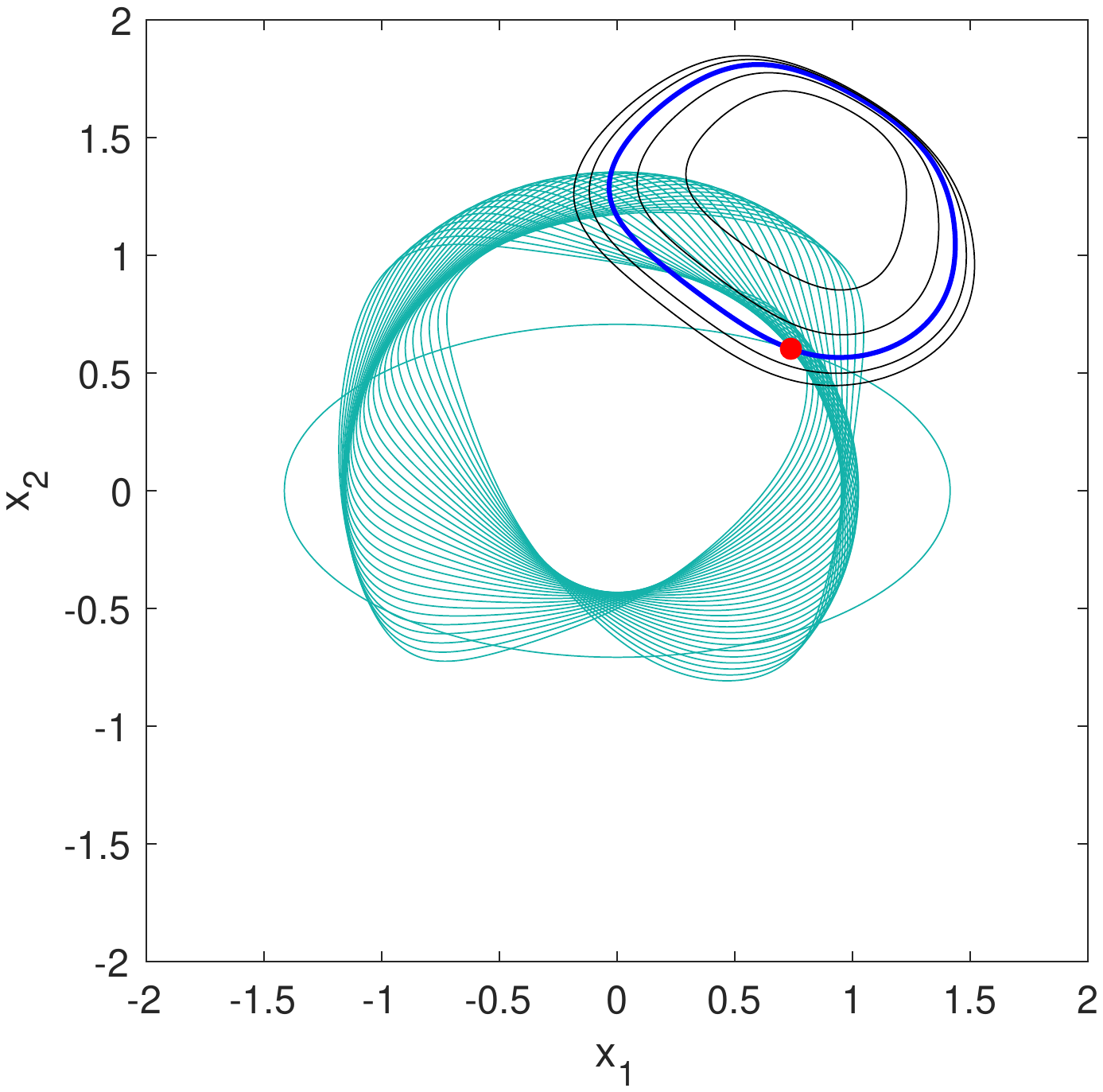}
	}
	\caption{The feasible set $\K$ and contour curves of $f/g$ in
		Example \ref{ex::sym}\label{fig::exsym} }
\end{figure}

Next, we apply the ACA method to \eqref{eq::ex1}.
It turns out that
the first phase of ACA to find a consistent initial approximation of
\eqref{eq::ex1} with a reduced outer approximation of $\Y$ always
failed. That is possibly because $\Y$ in \eqref{eq::ex1} has no
interior point
and the upper level problem is not convex (c.f. \cite{ACA2007,
ACA2012}). Then, we reformulate \eqref{eq::ex1} to the following
equivalent fractional semi-infinite programming problem involving
trigonometric functions and a single parameter $t$
\begin{equation}\label{eq::ex1r}
		\min_{x\in\RR^2}\  \frac{f(x)}{g(x)}\quad
		\text{s.t.}\ \psi(x)\le 0, \ p(\bx, \sin t, \cos t)\le 0,\ \
		\forall t\in [0,\pi/2].
\end{equation}
Then we solve \eqref{eq::ex1r} by the ACA method again. After a
successful phase I, the first 10 iterations of phase II to compute an
$\varepsilon$-stationary point of \eqref{eq::ex1r} run for about 22
minutes and produced a feasible point $(0.6530, 0.6272)$. The 15th
iteration of phase II produced a feasible point $(0.7374, 0.6034)$ and
the accumulated CPU time is about 40 minutes. The algorithm did not
reach its default termination criterion within an hour. 
}
\qed
\end{example}

\section{Special cases with exact or finitely convergent SDP
relaxations}\label{sec::SDP4}

In this section, we study some cases of the FSIPP problem \eqref{FMP}, 
for which we
can derive SDP relaxation which is exact or has finite convergence
and can extract at least one minimizer of \eqref{FMP}. The reason for
this concern is due to some applications of the FSIPP problem where
exact optimal values and minimizers are required, see Section \ref{sec::mfsipp}.

Recall the reformulations \eqref{eq::f*r} and \eqref{eq::f*rdual}.
Letting $\varepsilon=0$ in Theorem \ref{th::main} implies that
\begin{theorem}\label{th::mainOP}
	Suppose that {\sf (A1-3)} hold and the
	cones $\cs$ and $\cy$ satisfy the following conditions$:$ 
$\cy\subseteq\mathscr{P}(\Y)$, 
	$L_{f,g}(x,\mu^{\star},\eta^{\star})\in\cs,$
	there exists some $u^{\star}\in\bs$ such that
	$-p(u^{\star},y)\in\cy$ and $h(u^{\star})\ge 0$
	for any $h(x)\in\cs$.
	\begin{enumerate}[\upshape (i)]
		\item We have $r^{\psdp}=r^{\dsdp}=r^{\star}$.
\item If $\mathscr{L}^{\star}$ is a minimizer of \eqref{eq::f*rdual}
	such that the restriction $\mL^{\star}|_{\RR[x]_{\mathbf{d}}}$
	admits a representing nonnegative measure $\nu,$ then 
	\[
		\frac{\mathscr{L}^{\star}(\bX)}{\mathscr{L}^{\star}(1)}=
		\frac{1}{\int d\nu}\left(\int x_1 d\nu,\ldots,\int x_m
		d\nu\right)\in \bs. 
	\]
	\end{enumerate}
\end{theorem}

Next, we specify four cases of the FSIPP problem, for which we can
choose suitable cones $\cs$ and $\cy$ with sum-of-squares structures
and satisfy conditions in Theorem \ref{th::mainOP}.

\subsection{Four cases}
Recall the s.o.s-convexity introduced in Section \ref{sec::pre} and
consider
\begin{enumerate}[\text{\sf Case} 1.]
	\item\label{case1} (i) $n=1$ and $\Y=[-1,1]$; (ii) $f(\bX)$,
		$-g(x)$, $\psi_i(x), i=1,\ldots,s$, and
		$p(\bX,y)\in\RR[\bX]$ for every
		$y\in \Y$ are all {s.o.s-convex} in $\bX$.
	\item\label{case2} (i) $n>1$, $\Y=\{y\in\RR^n\mid\phi(y)\ge 0\}$
		where $\deg(\phi(y))=2$, $\phi(\bar{y})>0$ for some
		$\bar{y}\in\RR^n$; (ii) $\deg_y(p(x,y))=2$; (iii) $f(\bX)$, $-g(x)$,
		$\psi_i(x), i=1,\ldots,s$, and $p(\bX,y)\in\RR[\bX]$ for
		every $y\in \Y$ are all {s.o.s-convex} in $\bX$.
\end{enumerate}

Let $d_{\bX}=\deg_{\bX}(p(\bX,\bY))$ and
$d_{\bY}=\deg_{\bY}(p(\bX,\bY))$.
For {\sf Case} 1 and {\sf Case} 2, we make the following choices of $\cs$
and $\cy$ in the reformulations \eqref{eq::f*r} and
\eqref{eq::f*rdual}:
\begin{enumerate}[\text{In \sf Case} 1:]
	\item Let
		\begin{equation}\label{eq::crm}
			\cs=\Sigma^2[x]\cap\RR[x]_{2\mathbf{d}},
		\end{equation}
		and
		\begin{equation}\label{eq::cs}
				\cy=\left\{\theta_0+\theta_1(1-y_1^2)\
					\Big|\
					\begin{aligned}
						&\theta_0,\theta_1\in\Sigma^2[y_1],
						\deg(\theta_0)\le 2\lceil d_y/2\rceil, \\
				&\deg(\theta_1(1-y_1^2))\le 2\lceil d_y/2\rceil
			\end{aligned}\right\}.
		\end{equation}
	\item Let $\cs$ be defined as in \eqref{eq::crm}
		and
		\begin{equation}\label{eq::cs2}
			\cy=\{\theta+\lambda \phi\mid \lambda\ge 0,\
			\theta\in\Sigma^2[y],\ \deg(\theta)\le 2\}.
		\end{equation}
\end{enumerate}

Recall Proposition \ref{prop::eq} and Remark \ref{rk::eq}. In {\sf
Case} 1 and 2, we in fact choose $\mathcal{X}=\RR^m$ and
$\Sigma^2[x]\cap\RR[x]_{2\mathbf{d}}$ as the approximation of
$\mathscr{P}(\mathcal{X})$. 
In each case, we can reduce \eqref{eq::f*r} and
\eqref{eq::f*rdual} to a pair of primal and dual SDP problems.

\begin{lemma}\label{lem::sosconvex2}
	Under {\sf (A1-2)}$,$	if $f(\bX)$, $-g(x),$ $\psi_i(x),\
	i=1,\ldots,s,$ and $p(\bX,y)\in\RR[\bX]$ for
		every $y\in \Y$ are all {s.o.s-convex} in $\bX,$ then the
		Lagrangian $L_{f,g}(x,\mu,\eta)$ is s.o.s-convex for any
		$\mu\in\mathcal{M}(\Y)$ and $\eta\in\RR_+^s.$
\end{lemma}
\begin{proof}
Obviously, we only need to prove that $\int_{\Y} p(x,y)d\mu(y)$ is
s.o.s-convex under {\sf (A1-2)}. Note that there is a sequence of
atomic measures
$\{\mu_k\}\subseteq\mathcal{M}(\Y)$ which is weakly convergent to
$\mu$, i.e., $\lim_{k\rightarrow\infty} \int_{\Y}
h(y)d\mu_k(y)=\int_{\Y} h(y)d\mu(y)$ holds for every bounded
continuous real function $h(y)$
on $\Y$ (c.f. \cite[Example 8.1.6 (i)]{Bogachev}). It is obvious that
$\int_{\Y} p(x,y)d\mu_k(y)\in\RR[x]_{d_x}$ is s.o.s-convex for each $k$.
Since the convex cone of s.o.s-convex polynomials in $\RR[x]_{d_x}$ is
closed (c.f. \cite{Ahmadi2012}), the conclusion follows.%\qed
\end{proof}

\begin{theorem}\label{th::sdp12}
	In {\sf Cases} $1$-$2:$
	under {\sf (A2)}$,$ 
	the following holds.

		\begin{enumerate}[\upshape (i)]
			\item 
				$r^{\dsdp}=r^{\star}$ and
				$\frac{\mL^{\star}(x)}{\mL^{\star}(1)}\in\bs$ 
				where $\mL^{\star}$ be a minimizer of
				\eqref{eq::f*rdual} which always exists.
			\item If {\sf (A3)} holds$,$ then $r^{\psdp}=r^{\star}$
				and it is attainable.
	\end{enumerate}
\end{theorem}
\begin{proof}
	In {\sf Case} 1, by the representations of univariate polynomials
	nonnegative on an interval (c.f. \cite{Laurent_sumsof,PR2000}), we
	have $-p(x,y)\in\cy$ for each $x\in\K$. In {\sf Case} 2, by the
	$S$-lemma and Hilbert's theorem, we also have
	$-p(x,y)\in\cy$ for each $x\in\K$.
	For any $u^{\star}\in\bs$,  the linear functional $\mL'\in(\RR[x])^*$ such that
$\mL'(x^\alpha)=\frac{(u^{\star})^\alpha}{g(u^{\star})}$ for each
$\alpha\in\N^m$, is feasible to \eqref{eq::f*rdual}. Hence,
$r^{\psdp}\le r^{\dsdp}\le r^{\star}$ by the weak duality.

(i) Let $\mL^{\star}$ be a minimizer of \eqref{eq::f*rdual}, then
$\mL^{\star}(1)>0$. In fact, $\mL^{\star}(1)\ge 0$ since
$\mL^{\star}\in(\Sigma^2[x]\cap\RR[x]_{2\mathbf{d}})^*$. If
$\mL^{\star}(1)=0$, then by the positive semidefiniteness of the
associated moment matrix of $\mL^{\star}$, we have $\mL^{\star}(x^{\alpha})=0$
for all $\alpha\in\N^m_{\mathbf{d}}$, which contradicts the equality
$\mL^{\star}(g)=1$. 
As $\psi_1(x),\ldots,\psi_s(x)$, $p(\bX,y)\in\RR[\bX]$
 for every $y\in \Y$ are all s.o.s-convex in $\bX$, similar to the
 proof of Theorem \ref{th::main} (ii), it is easy to see that
 $\frac{\mL^{\star}(x)}{\mL^{\star}(1)}\in\K$ due to Proposition
 \ref{prop::Jensen}.
 Since $f(\bX)$ and $-g(x)$ are also s.o.s-convex, under {\sf (A2)},
we have
\[
	r^{\star}\le\frac{f\left(\frac{\mL^{\star}(x)}{\mL^{\star}(1)}\right)}{g\left(\frac{\mL^{\star}(x)}{\mL^{\star}(1)}\right)}
	\le
	\frac{\frac{1}{\mL^{\star}(1)}\mL^{\star}(f)}{\frac{1}{\mL^{\star}(1)}\mL^{\star}(g)}=\mL^{\star}(f)=r^{\dsdp}\le r^{\star}.
\]
It means that $r^{\dsdp}=r^{\star}$ and $\frac{\mL^{\star}(x)}{\mL^{\star}(1)}$ is minimizer of \eqref{FMP}.
Clearly, for any $u^{\star}\in\bs$, the linear functional
$\mL'\in(\RR[x]_{2\mathbf{d}})^*$ such that
$\mL'(x^\alpha)=\frac{(u^{\star})^\alpha}{g(u^{\star})}$
for each $\alpha\in\N^m_{2\mathbf{d}}$ is a minimizer of \eqref{eq::f*rdual}. 

	(ii)
	By Lemma \ref{lem::sosconvex2}, Proposition \ref{prop::red} and
	Lemma \ref{lem::sosconvex},
	$L_{f,g}(x,\mu^{\star},\eta^{\star})\in\cs$ in both cases.
	Thus, $r^{\psdp}=r^{\star}$ due to Theorem \ref{th::mainOP} (i)
	and is attainable at $(r^{\star}, \mH, \eta^{\star})$ where
	$\mH\in(\RR[y])^*$ satisfies that $\mH(y^\beta)=\int_{\Y} y^\beta
	d\mu^{\star}(\by)$ for any $\beta\in\N^n$.
\end{proof}

\vskip 5pt

Now we consider another two cases of the FSIPP problem \eqref{FMP}:
\begin{enumerate}[\text{\sf Case} 1.]
		\setcounter{enumi}{2}
	\item (i) $n=1$ and $\Y=[-1,1]$; (ii) The Hessian $\nabla^2
		f(u^{\star})\succ 0$ at some $u^{\star}\in\bs.$
	\item (i) $n>1$, $\Y=\{y\in\RR^n\mid\phi(y)\ge 0\}$
		where $\deg(\phi(y))=2$, $\phi(\bar{y})>0$ for some
		$\bar{y}\in\RR^n$; (ii)
		$\deg_y(p(x,y))=2$; (iii) The Hessian $\nabla^2
		f(u^{\star})\succ 0$ at some $u^{\star}\in\bs$.
\end{enumerate}

Let $R>0$ be a real number satisfying \eqref{eq::rg} and
$q(x):=R^2-(x_1^2+\cdots+x_m^2)$. 
For an integer $k\ge\lceil \mathbf{d}/2\rceil$,
we make the following choices of $\cs$ and
$\cy$ in the reformulations \eqref{eq::f*r} and
\eqref{eq::f*rdual}
in {\sf Case} 3 and {\sf Case} 4:
\begin{enumerate}[\text{In \sf Case} 1:]
		\setcounter{enumi}{2}
	\item Replace $\cs$ by $\csk{k}=\qm_k(\{q\})$
		and let $\cy$ be defined as in \eqref{eq::cs}.
	\item Replace $\cs$ by $\csk{k}=\qm_k(\{q\})$
		and let $\cy$ be defined as in \eqref{eq::cs2}.
\end{enumerate}

Recall Proposition \ref{prop::eq} and Remark \ref{rk::eq}. In {\sf
Case} 3 and 4, we in fact choose $\mathcal{X}=\{x\in\RR^m\mid q(x)\ge
0\}$ and quadratic modules generated by $\{q\}$ as the approximation
of $\mathscr{P}(\mathcal{X})$. 
For a fixed $k$, in each case, denote the resulting problems of
\eqref{eq::f*r} and \eqref{eq::f*rdual} by $(\ref{eq::f*r}k)$ and
$(\ref{eq::f*rdual}k)$, respectively.  Denote by $r^{\psdp}_k$ and $r^{\dsdp}_k$ the
optimal values of $(\ref{eq::f*r}k)$ and $(\ref{eq::f*rdual}k)$.
We can reduce $(\ref{eq::f*r}k)$ and
$(\ref{eq::f*rdual}k)$ to a pair of primal and dual SDP problems.

\begin{theorem}\label{th::sdp34}
	In {\sf Cases} $3$-$4:$ under {\sf (A1-2)}$,$ 
	the following holds.
	\begin{enumerate}[\upshape (i)]
		\item For each $k\ge\lceil \mathbf{d}/2\rceil,$
			$r^{\psdp}_k\le r^{\dsdp}_k\le r^{\star}$ and
			$r^{\dsdp}_k$ is attainale.
		\item For a minimizer $\mL_k^{\star}$ of
			$(\ref{eq::f*rdual}k)$, if 
			there exists an integer $k'\in[\lceil \mathbf{d}/2\rceil, k]$ such that
	\begin{equation}\label{eq::rk2}
		\rank\ \mathbf{M}_{k'-1}(\mL_k^{\star})=\rank\
		\mathbf{M}_{k'}(\mL_k^{\star}),
	\end{equation}
	then$,$ $r^{\dsdp}_k=r^{\star}$ and $\frac{\mL_k^{\star}(x)}{\mL_k^{\star}(1)}$ is
	minimizer of \eqref{FMP}$;$
\item 
	If {\sf (A3)}	holds$,$ then for $k$ large enough$,$
			$r_{k}^{\psdp}=r_{k}^{\dsdp}=r^{\star}$ and every minimizer $\mL_k^{\star}$ of
			$(\ref{eq::f*rdual}k)$
			satisfies the rank condition \eqref{eq::rk2} which
			certifies that $\frac{\mL_k^{\star}(x)}{\mL_k^{\star}(1)}$
			is a minimizer of \eqref{FMP}.
		\end{enumerate}
\end{theorem}
\begin{proof}
	(i) As proved in Theorem \ref{th::sdp12}, we have $-p(u^{\star},y)\in\cy$ for
	every $u^{\star}\in\bs\subset\K$ in both cases and hence $r^{\psdp}_k\le
	r^{\dsdp}_k\le r^{\star}$. Due to the form of $q(x)$, the
	attainment of $r^{\dsdp}_k$ follows from \cite[Lemma 3]{CD2016} as
	proved in Theorem \ref{th::asy} (ii). 

	(ii) From the proof of Theorem \ref{th::asy} (iii), we can see that
	$\mL_k^{\star}(1)>0$. By \cite[Theorem 1.1]{CFKmoment}, 
	\eqref{eq::rk2} implies that the restriction
$\mL_k^{\star}|_{\RR[x]_{2k'}}$ has an atomic representing measure
supported on the set $\K':=\{\bx\in\RR^m\mid q(x)\ge 0\}$.
Then, the conclusion follows by Theorem \ref{th::mainOP} (ii).

(iii)
Under {\sf (A1-3)}, consider the nonnegative Lagrangian $L_{f,g}(x,\mu^{\star},\eta^{\star})$.
By Proposition \ref{prop::sc},
$L_{f,g}(x,\mu^{\star},\eta^{\star})$ is coercive and strictly convex on $\RR^m$.
Hence, by Proposition \ref{prop::red}, $\bs$ is a singleton set, say
$\bs=\{u^{\star}\}$, and $u^{\star}$ is the unique minimizer of
$L_{f,g}(x,\mu^{\star},\eta^{\star})$ on $\RR^m$.
Clearly, $u^{\star}$ is an interior point of $\K'$.
Then, by Proposition \ref{prop::s},
there exists
$k^{\star}\in\N$ such that
$L_{f,g}(x,\mu^{\star},\eta^{\star})\in\qm_{k^{\star}}(\{q\})$.
Thus, in both {\sf Case} 3 and {\sf Case} 4,
$L_{f,g}(x,\mu^{\star},\eta^{\star})\in\csk{k}$ for every
$k\ge k^{\star}$. Then, $r^{\psdp}_{k}=r^{\dsdp}_{k}=r^{\star}$
for each $k\ge k^{\star}$ by Theorem \ref{th::mainOP} (i).

Consider the polynomial optimization problem
\begin{equation}\label{eq::lp}
		l^{\star}:=\min_{\bx\in\RR^m}\ L_{f,g}(\bx,\mu^{\star},\eta^{\star})\quad
		\text{s.t.}\quad q(x)\ge 0.
\end{equation}
Then, $l^{\star}=0$ and is attained at $u^{\star}$. 
The $k$-th Lasserre's relaxation (see Section \ref{sec::pre}) for
\eqref{eq::lp} is 
\begin{equation}\label{eq::lpp}
l_k^{\dsdp}:=\inf_{\mL}\ \mL(L_{f,g}(x,\mu^{\star},\eta^{\star}))\quad 
\text{s.t.}\ \ \mL\in(\qm_{k}(\{q\}))^*,\ \mL(1)=1,
\end{equation}
and its dual problem is
\begin{equation}\label{eq::lpd}
l_k^{\psdp}:=\sup_{\rho\in\RR}\ \rho\quad \text{s.t.}\ \
L_{f,g}(x,\mu^{\star},\eta^{\star})-\rho\in\qm_k(\{q\}).
\end{equation}
We have shown that $l_{k^{\star}}^{\psdp}\ge 0$.
As the linear functional $\mL'\in(\RR[x]_{2k})^*$ with
$\mL'(x^{\alpha})=(u^{\star})^{\alpha}$ for each $\alpha\in\N^m_{2k}$ is
feasible to \eqref{eq::lpp}, along with the weak duality,
we have $l_k^{\psdp}\le l_k^{\dsdp}\le l^{\star}=0$, which means
$l_{k^{\star}}^{\psdp}=l_{k^{\star}}^{\dsdp}=0$.
Hence, Lasserre's hierarchy \eqref{eq::lpp} and \eqref{eq::lpd} have finite
convergence at the order $k^{\star}$ without dual gap and the optimal value
of \eqref{eq::lpd} is attainable. Moreover, recall that $u^{\star}$ is the unique
point in $\RR^m$ such that $L_{f,g}(u^{\star},\mu^{\star},\eta^{\star})=0=l^{\star}$.
Then by Proposition \ref{prop::fc}, the rank condition \eqref{eq::rk2} holds for
every minimizer of
\eqref{eq::lpp} for sufficiently large $k$. Let $\mL_k^{\star}$ be a
minimizer of $(\ref{eq::f*rdual}k)$ with $k\ge k^{\star}$. 
Now we show that  $\frac{\mL_k^{\star}}{\mL_k^{\star}(1)}$ is a
minimizer of \eqref{eq::lpp}.
Clearly, $\frac{\mL_k^{\star}}{\mL_k^{\star}(1)}$ is feasible to \eqref{eq::lpp}.
Because
\[
	\begin{aligned}
		0&=l^{\star}=l_k^{\dsdp}\le\frac{\mL_k^{\star}(L_{f,g}(x,\mu^{\star},\eta^{\star}))}{\mL_k^{\star}(1)}\\
		&=\frac{1}{\mL_k^{\star}(1)}\left(\mL_k^{\star}(f)-
		r^{\star}\mL_k^{\star}(g)
		+\int_{\Y}\mL_k^{\star}(p(\bx,y))d\mu^{\star}+\sum_{j=1}^s\eta_j\mL_k^{\star}(\psi_j)\right)\\
		 &=\frac{1}{\mL_k^{\star}(1)}\left(r^{\star}-r^{\star}+\int_{\Y}\mL_k^{\star}(p(\bx,y))d\mu^{\star}
		 +\sum_{j=1}^s\eta_j\mL_k^{\star}(\psi_j)\right)\le 0,\\
	\end{aligned}
\]
$\frac{\mL_k^{\star}}{\mL_k^{\star}(1)}$ is indeed a minimizer of \eqref{eq::lpp}. 
Therefore, for $k$ sufficiently large, the rank condition \eqref{eq::rk2} holds for
$\frac{\mL_k^{\star}}{\mL_k^{\star}(1)}$ and hence for $\mL_k^{\star}$.
\end{proof}

\begin{example}\label{ex::4cases}{\rm
Now we consider four FSIPP problems corresponding to the four cases
studied above.
	
		\vskip 5pt
\noindent{\sf Case} 1: 
		Consider the FSIPP problem
\begin{equation}\label{eq::case1} 
	\left\{ \begin{aligned}
		\min_{x\in\RR^2}\
		&\frac{(x_1+1)^2+(x_2+1)^2}{-x_1-x_2+1}\\ \text{s.t.}\
		& p(x,y)=x_1^2+y^2x_2^2+2yx_1x_2+x_1+x_2\le 0 ,\\ &
		\forall\ y\in [-1, 1].  \end{aligned} \right.
\end{equation} 
For any $y\in[-1,1]$, since $p(x,y)$ is of
degree $2$ and convex in $x$, it is s.o.s-convex in $x$.
Hence, the problem \eqref{eq::case1} is in {\sf Case} 1.  For
any $x\in\RR^2$ and $y\in[-1, 1]$, it is clear that \[
p(x,y)\le  x_1^2+x_2^2+2|x_1x_2|+x_1+x_2.  \] Then we can see
that the feasible set $\K$ can be defined only by two
constraints 
\[ p(x,1)=(x_1+x_2)(x_1+x_2+1)\le 0\ \text{and}\
p(x,-1)=(x_1-x_2)^2+x_1+x_2\le 0.  
\] 
That is, $\K$ is the area in $\RR^2$ enclosed by the ellipse
$p(x,-1)=0$ and the
two lines $p(x,1)=0$. Then, it is not hard to check that the
only global minimizer of \eqref{eq::case1} is
$u^{\star}=(-0.5,-0.5)$ and the minimum is $0.25$. Obviously,
{\sf (A2)} holds for \eqref{eq::case1}. Solving the single SDP
problem \eqref{eq::f*rdual} with the setting \eqref{eq::crm}
and \eqref{eq::cs}, we get
$\frac{\mL^{\star}(x)}{\mL^{\star}(1)}=(-0.5000, -0.5000)$
where $\mL^{\star}$ is the minimizer of \eqref{eq::f*rdual}.
The CPU time is $0.80$ seconds.
Then we solve \eqref{eq::case1} with ACA method. The algorithm
terminated successfully and returned the solution $(-0.5000,
-0.5000)$. The over CPU time is $4.25$ seconds.

		\vskip 5pt
\noindent{\sf Case} 2: 
Consider the FSIPP problem
\begin{equation}\label{eq::case2}
	\left\{
	\begin{aligned}
		\min_{x\in\RR^2}\ &\frac{(x_1-1)^2+(x_2-1)^2}{x_1+x_2}\\
		\text{s.t.}\ & \psi(x)=(x_1+x_2-1)(x_1+x_2-0.5)\le 0,\\
		&	p(x,y)=(y_1^2+y_2^2)x_1^2+(1/2-y_1y_2)x_2^2-1\le 0 ,\ \forall\
		y\in\Y,\\
		& \Y=\{y\in\RR^2\mid y_1^2+y_2^2\le 1\}. 
	\end{aligned}
	\right.
\end{equation}
It is easy to see that \eqref{eq::case2} is in {\sf Case} 2. 
For any $y\in\Y$, it holds that 
\[
	p(x,y)\le x_1^2+x_2^2-1=p(x,y^{(0)}), \quad
	y^{(0)}=\left(-\frac{\sqrt{2}}{2}, \frac{\sqrt{2}}{2}\right). 
\]
Hence, $\K$ is the part of the unit disc around the origin between the
two lines defined by $\psi(x)=0$ and the only global minimizer is
$u^{\star}=(0.5, 0.5)$. Obviously, {\sf (A2)} holds for
\eqref{eq::case2}. Solving the single SDP problem \eqref{eq::f*rdual}
with the setting \eqref{eq::crm} and \eqref{eq::cs2}, we get
$\frac{\mL^{\star}(x)}{\mL^{\star}(1)}=(0.4999, 0.5000)$ where
$\mL^{\star}$ is the minimizer of \eqref{eq::f*rdual}. The CPU time is
$1.20$ seconds. 
Then we solve \eqref{eq::case2} with ACA method. The algorithm
terminated successfully and returned the solution $(0.5000,
0.5000)$. The overall CPU time is $52.75$ seconds.

		\vskip 5pt
\noindent{\sf Case} 3: 
Recall the convex but not s.o.s-convex polynomial $h_2(x)$ in
\eqref{eq::nonsosconvex}.  Consider the FSIPP problem
\begin{equation}\label{eq::case3}
	\left\{
	\begin{aligned}
		\min_{x\in\RR^2}\ &\frac{(x_1-1)^2+(x_2-1)^2-2}{-x_1-x_2+4}\\
		\text{s.t.}\ & \psi(x)=x_1^2+x_2^2-4\le 0,\\
	&	p(x,y)=\frac{h_2(x_1,x_2)}{1000}-yx_1-y^2x_2-1\le 0 ,\ \forall\
		y\in [-1 ,1].
	\end{aligned}
	\right.
\end{equation}
Clearly, this problem is in {\sf Cases} $3$ and satisfies {\sf
(A1-2)}. We solve the SDP relaxation $(\ref{eq::f*rdual}k)$ 
with the setting $\csk{k}$ and $\cy$ aforementioned. 
We set the first order $k=4$ and check if
the rank condition \eqref{eq::rk2} holds. If not, check the next order. 
We have  $\rank
\mathbf{M}_{3}(\mL_4^{\star})=\rank \mathbf{M}_{4}(\mL_4^{\star})=1$
(within a tolerance $<10^{-8}$) for a minimizer $\mL_4^{\star}$ of
of $r^{\dsdp}_4$, i.e., the rank condition \eqref{eq::rk2} holds for
$k'=4$. By Theorem \ref{th::sdp34} (ii), the point
$u^{\star}:=\frac{\mL_4^{\star}(x)}{\mL_4^{\star}(1)}=(0.9044,
0.8460)$ is a minimizer and $r^{\dsdp}_4=-0.8745$ is the minimum of
\eqref{eq::case3}. The CPU time is about $16.50$ seconds.
To show the accuracy of the solution , we draw some contour curves of
$f/g$, including the one where $f/g$ is the constant value
$f(u^{\star})/g(u^{\star})=-0.8745$ (the blue curve), and mark the point
$u^{\star}$ by a red dot in Figure \ref{fig::case34} (left). 
Then we solve \eqref{eq::case3} with the ACA method. The algorithm
terminated successfully and returned the solution $(0.9040, 0.8463)$.
The overall CPU time is $21.57$ seconds.

		\vskip 5pt
\noindent{\sf Case} 4: Consider the FSIPP problem
\begin{equation}\label{eq::case4}
	\left\{
	\begin{aligned}
		\min_{x\in\RR^2}\ &\frac{(x_1-2)^2+(x_2-2)^2-1}{-x_1^2-x_2^2+4}\\
		\text{s.t.}\ & \psi(x)=x_1^2+x_2^2-1\le 0,\\
		& p(x,y)=\frac{h_2(x_1,x_2)}{1000}+y_1y_2(x_1+x_2)-1\le 0 ,\ \forall\
		y\in\Y,\\
		& \Y=\{y\in\RR^2\mid y_1^2+y_2^2\le 1\}. 
	\end{aligned}
	\right.
\end{equation}
Clearly, this problem is in {\sf Cases} $4$ and satisfies {\sf
(A1-2)}. 
We solve the SDP relaxation $(\ref{eq::f*rdual}k)$ 
with the setting $\csk{k}$ and $\cy$ aforementioned. 
For the first order $k=4$, we have  $\rank
\mathbf{M}_{3}(\mL_4^{\star})=\rank \mathbf{M}_{4}(\mL_4^{\star})=1$
(within a tolerance $<10^{-8}$) for a minimizer $\mL_4^{\star}$ of
of $r^{\dsdp}_4$. By Theorem \ref{th::sdp34} (ii), the point
$u^{\star}:=\frac{\mL_4^{\star}(x)}{\mL_4^{\star}(1)}=(0.7211,
0.6912)$ is a minimizer of \eqref{eq::case4}. The CPU time is about
$14.60$ seconds. See Figure \ref{fig::case34} (right) for the accuracy
of the solution. 
Then we solve \eqref{eq::case4} with the ACA method. The algorithm
terminated successfully and returned the solution $(0.7039, 0.6823)$.
The overall CPU time is $768.02$ seconds.
\vskip 5pt

For the above four FSIPP problems, 
we remark that the optimality of the solution obtained by our SDP
method can be guaranteed by Theorem \ref{th::sdp12} (i) and Theorem
\ref{th::sdp34} (ii), while the solution concept of the ACA method is
that of stationary points and all iterates are feasible points for the
original SIP.
\qed

\begin{figure}
	\centering
	\scalebox{0.5}{
		\includegraphics[trim=80 200 80 200,clip]{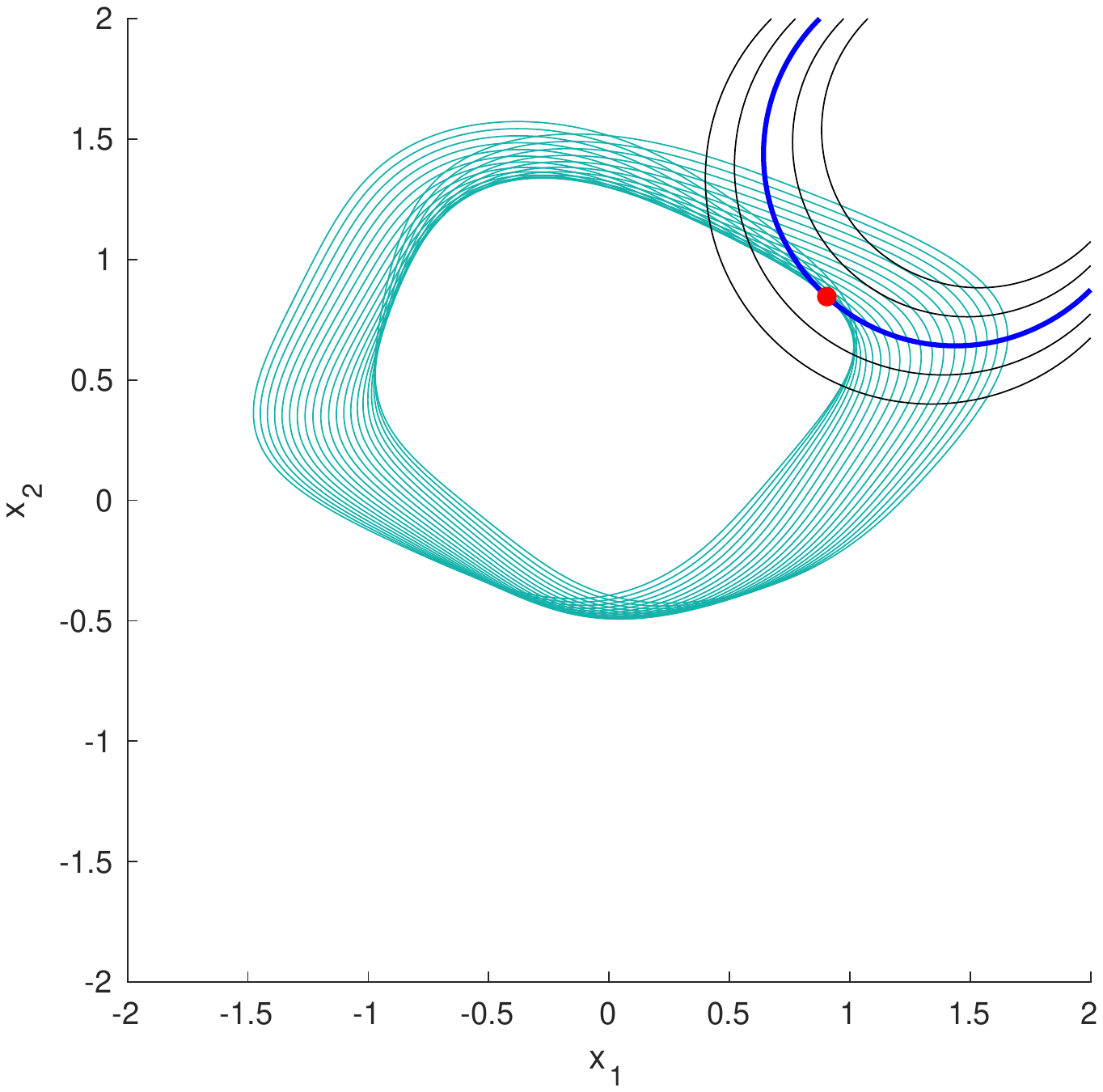}
	}
	\scalebox{0.5}{
		\includegraphics[trim=80 200 80 200,clip]{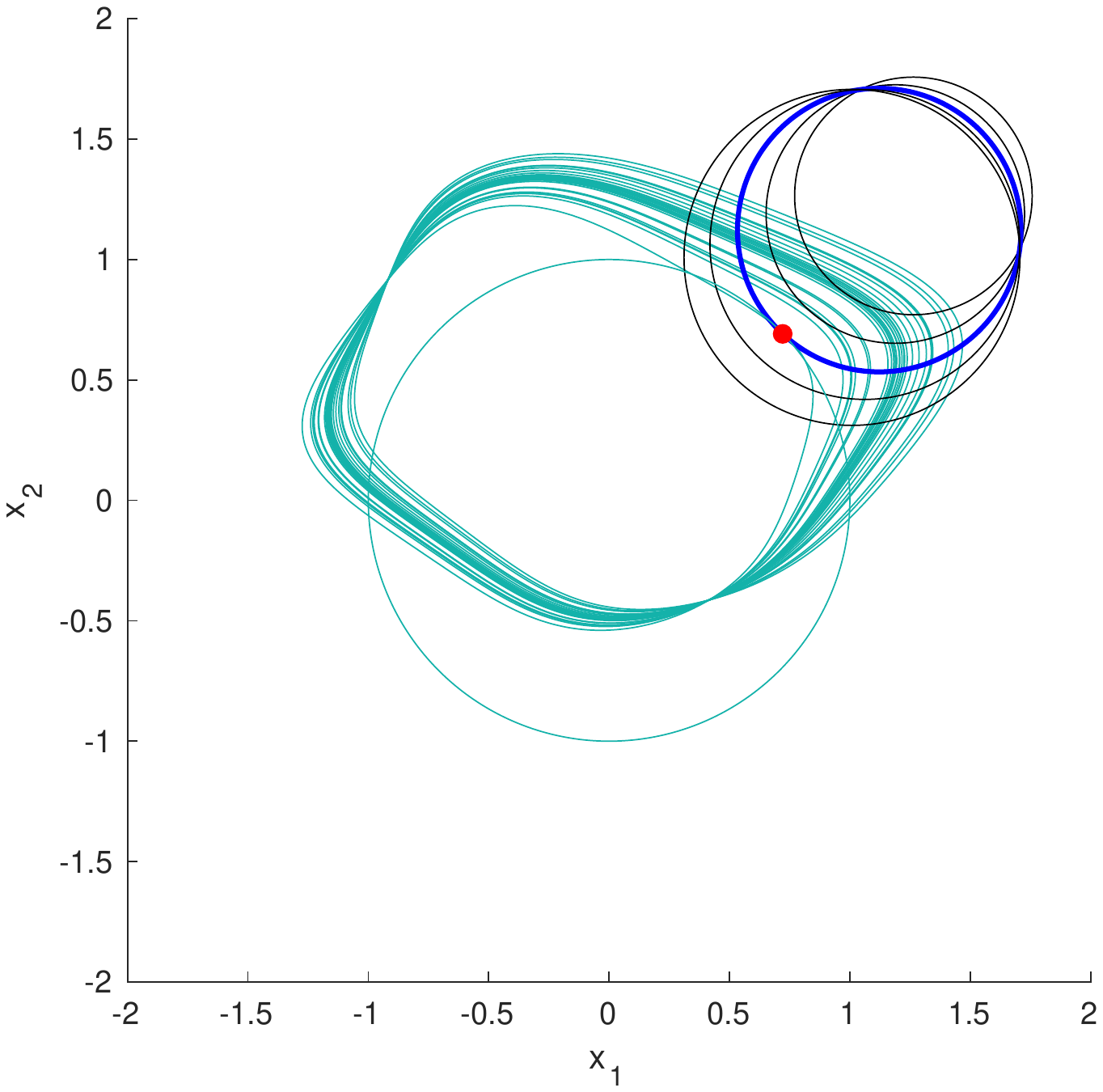}
	}
	\caption{The feasible set $\K$ and contour curves of $f/g$ in
		Example \ref{ex::4cases} {\sf Case} 3 (left) and {\sf Case} 4
		(right). \label{fig::case34} }
\end{figure}
}
\end{example}

\subsection{Application to Multi-objective FSIPP}\label{sec::mfsipp}

In this part, we apply the above approach for the special four cases
of FSIPP problems  
to the following multi-objective fractional semi-infinite polynomial
programming (MFSIPP) problem
\begin{equation}\label{MFP}
	\left\{
	\begin{aligned}
		{\rm Min}_{\RR^t_+}&\  \left(\frac{f_1(x)}{g_1(x)},\ldots,\frac{f_t(x)}{g_t(x)}\right)\\
		\text{s.t.}&\ p(\bx,\by)\le 0,\ \ \forall \by\in
		\Y\subset\RR^n,
	\end{aligned}\right.
\end{equation}
where
$f_i(\bX), g_i(x)\in\RR[\bX],$ $i = 1, \ldots, t,$ $p(\bX,\bY)
\in\RR[\bX,\bY]$.
Note that ``${\rm Min_{\RR^t_+}}$'' in the above
problem~\eqref{MFP} is understood in the vectorial sense, where a
partial ordering is induced in the image space $\RR^t,$ by the
non-negative cone $\RR^t_+.$
Let $a, b \in \RR^t,$ the partial ordering says that $a \geq b$ (or $a
- b \in \RR^t_+$), which can equivalently be written as $a_i \geq b_i,$
for all $i = 1, \ldots, t,$ where $a_i$ and $b_i$ stands for the $i$th
component of the vectors $a$ and $b,$ respectively.
Denote by $\F$ the feasible set of \eqref{MFP}.
We make the following assumptions on the MFSIPP problem \eqref{MFP}:

\begin{quote}
	{\sf (A5):} $\Y$ is compact; $f_i(\bX)$, $-g_i(x)$, $i = 1, \ldots, t,$ and
	$p(\bX,\by)\in\RR[\bX]$ for every $\by\in \Y$ are all convex in
	$\bX$;\\
	{\sf (A6):} For each $i=1,\ldots,t$, either $f_i(x)\ge 0$ and
	$g_i(x)>0$ for all $x\in\F$; or $g_i(x)$ is affine and $g_i(x)>0$ for
	all $x\in\F$.
\end{quote}

\begin{definition}\label{def::es}
	A point $u^{\star}\in \F$ is said to be an {\it efficient solution} to
	\eqref{MFP} if
	\begin{equation}\label{efficientsolution}
		\left(\frac{f_1(x)}{g_1(x)},\ldots,\frac{f_t(x)}{g_t(x)}\right)-
		\left(\frac{f_1(u^{\star})}{g_1(u^{\star})},\ldots,\frac{f_t(u^{\star})}{g_t(u^{\star})}\right)
		\not\in-\RR^{t}_+\backslash\{0\}, \quad \forall x\in \F.
	\end{equation}
\end{definition}
Efficient solutions to \eqref{MFP} are also known as Pareto-optimal
solutions. The aim of this part is to find efficient solutions to
\eqref{MFP}.
As far as we know, very few algorithmic developments
are available for such a case in the literature because of the
difficulty of checking feasibility of a given point.

The $\epsilon$-constraint method \cite{Econstraint,Chankong1983} may be the best
known technique to solve a {\it nonconvex} multi-objective optimization problem.
The basic idea for this method is to minimize one of the original objectives while
the others are transformed to constraints by setting an upper bound to
each of them. 
Based on the criteria for the $\epsilon$-constraint method given in
\cite{Ehrgott2005}, an algorithm to obtain an efficient solution to
\eqref{MFP} follows.
\begin{framed}
\begin{algorithm}\label{al::main}{\rm (Compute an efficient solution
		$u^{\star}$ to the MFSIPP problem \eqref{MFP}.)}
\vskip -10pt
		{\sf 
\begin{enumerate}[\upshape 1.]
	\item Set $i=1$ and choose an initial point $u^{(i-1)}\in\F$.
	\item Solve the single objective FSIPP problem
		\begin{equation}\label{eq::p}
	r_i:=\min_{x\in\F}\ \frac{f_i(x)}{g_i(x)}\quad \text{s.t.}\	\
	g_j(u^{(i-1)})f_j(x)-f_j(u^{(i-1)})g_j(x)\le 0,\ \ j\neq i.
	\tag{{$\mathrm{P}_i$}}
		\end{equation}
		and extract a minimizer $u^{(i)}$ of \eqref{eq::p}.
	\item If either $u^{(i)}$ can be verified to be the unique
		minimizer of \eqref{eq::p} or $i=t$, return $u^{\star}=u^{(i)}$;
		otherwise, let $i=i+1$ and go to Step 2.
\end{enumerate}}
\end{algorithm}
\end{framed}
\begin{theorem}\label{th::correctness}
	The output $u^{\star}$ in Algorithm \ref{al::main} is indeed an
	efficient solution to \eqref{FMP}.
\end{theorem}
\begin{proof}
We refer to \cite[Propositions~4.4 and 4.5]{Ehrgott2005}; see also \cite[Theorem 3.4]{Lee2018}. 
\end{proof}
\begin{remark}\label{rk::slaterfails}{\rm
\begin{itemize}
  \item[{\rm (i)}] Clearly, \eqref{eq::p} is an FSIPP problem of the form \eqref{FMP}.
It is easy to see that for each $i=2,\ldots,t$, the constraints
\[
	g_j(u^{(i-1)})f_j(x)-f_j(u^{(i-1)})g_j(x)\le 0,\ j=1,\ldots,i-1,
\]
are all active in \eqref{eq::p}. Therefore, the
Slater condition fails for \eqref{eq::p} with $i=2,\ldots,t.$
  \item[{\rm (ii)}] According to Algorithm \ref{al::main}, the problem of finding an
	efficient solution of the MFSIPP problem \eqref{MFP} reduces to
	solving every scalarized problem \eqref{eq::p} and extracting a (common)
	minimizer, which is the key for the success of Algorithm \ref{al::main}.
	Generally, approximate solutions to \eqref{eq::p} can be obtained by some
	numerical methods for semi-infinite programming problems.
	However, note that
	the errors introduced by any approximate solutions can accumulate
	in the process of the $\epsilon$-constraint method. 
This can potentially make the output solution unreliable.\qed
\end{itemize}
}\end{remark}

We have studied four cases of the FSIPP problem, for which
at least one minimizer can be extracted by the proposed SDP approach. 
Now we apply this approach to the four corresponding cases of MFSIPP
problem: 

{\sf Case} I (resp., II):\ \ \eqref{eq::p} is in {\sf Case} 1 (resp., 2) for each $i=1,\ldots,t$;

{\sf Case} III (resp., IV):\ \ ($\mathrm{P}_{i'}$) is in {\sf Case} 3
(resp., 4) for some $i'\in\{1,\ldots,t\}$.

For {\sf Case} I and II, 
if the assumptions
in Theorem \ref{th::sdp12} hold for each \eqref{eq::p}, 
then an efficient solution to the MFSIPP problem
\eqref{MFP} can obtained by solving $t$ SDP problems.

For {\sf Case} III and IV, 
we only need solve ($\mathrm{P}_{i'}$)
to get an efficient solution to \eqref{MFP}.  In fact, we have the following result.
\begin{proposition}\label{prop::unique2}
In {\sf Cases} {\upshape III}-{\upshape IV}$:$ under {\sf (A5-6)}$,$
the scalarized problem $(\mathrm{P}_{i'})$ has a unique minimizer
$u^{(i')}$ which is an efficient solution to the MFSIPP
	problem \eqref{MFP}.
\end{proposition}
\begin{proof}
	By assumption, $f_{i'}(x)-r_{i'}g_{i'}(x)$ is convex and its minimum on the
feasible set of ($\mathrm{P}_{i'}$) is $0$ attained at any optimal
solution of ($\mathrm{P}_{i'}$).
By Proposition \ref{prop::sc}, $f_{i'}(x)-r_{i'}g_{i'}(x)$ is coercive and
strictly convex on $\RR^m$. Then,
$f_{i'}(x)-r_{i'}g_{i'}(x)$ has a unique minimizer on the feasible set
of ($\mathrm{P}_{i'}$). Consequently, ($\mathrm{P}_{i'}$)
has a unique minimizer $u^{(i')}$. By Theorem \ref{th::correctness},
$u^{(i')}$ is an efficient solution to \eqref{MFP}. %\qed
\end{proof}
As a result, in {\sf Case} III and {\sf Case} IV, if the assumptions
in Theorem \ref{th::sdp34} hold for ($\mathrm{P}_{i'}$),
an efficient solution to the MFSIPP problem \eqref{MFP} can be
obtained by solving finitely many SDP problems.
\vskip 5pt
\begin{example}{\rm
To show the efficiency of the SDP method for the four cases
of the MFSIPP problem discussed above, now we present an
example for each case. 
In each of the following examples, $m=2$ and $t=2$. We pick some
	points $y$ on a uniform discrete grid inside $\Y$ and draw the
	corresponding curves $p(x,y)=0$. Hence, the feasible set $\F$ is
	illustrated by the area enclosed by these curves.
The initial point $u^{(0)}$ and the output $u^{\star}$ of Algorithm
\ref{al::main} are marked in $\F$ by `$\ast$' in blue and red, respectively.
To show the accuracy of the output, we first illustrate the image of
$\F$ under the map $\left(\frac{f_1}{g_1},\frac{f_2}{g_2}\right)$. To
this end, we choose a square containing $\F$. For each point
$u$ on a uniform discrete grid inside the square, we check if
$u\in\F$ (as we will see it is easy for our examples).
If so, we plot the point
$\left(\frac{f_1(u)}{g_1(u)},\frac{f_2(u)}{g_2(u)}\right)$ in the
image plane.
The points
$\left(\frac{f_1(u^{(0)})}{g_1(u^{(0)})},\frac{f_2(u^{(0)})}{g_2(u^{(0)})}\right)$ and
$\left(\frac{f_1(u^{\star})}{g_1(u^{\star})},\frac{f_2(u^{\star})}{g_2(u^{\star})}\right)$
are then marked in the image by `$\ast$' in blue and red,
respectively. We will see from the figures that the output of Algorithm
\ref{al::main} in each example is indeed as we expect.

\vskip 5pt
\noindent{\sf Case} I:
	Consider the ellipse
	\[
		\F=\{(x_1,x_2)\in\RR^2\mid 2x_1^2+x_2^2+2x_1x_2+2x_1\le 0\},
	\]
	which can be represented by
	\[
		\{(x_1,x_2)\in\RR^2\mid p(x_1,x_2,y_1)\le 0,\ \forall y_1\in \Y\},
	\]
	where
	\[
		p(x_1,x_2,y_1)=(y_1^4+2y_1^3-3y_1^2-2y_1+1)x_1+2y_1(y_1^2-1)x_2-2y_1^2,
	\]
	and $\Y=[-1,1]$ $($See {\upshape\cite{LSIP}}$)$. The feasible set
	$\F$ is illustrated in Figure \ref{fig::case1} (left).

	Consider the problem
	\[
		{\rm Min}_{\RR^{2}_+}\left\{\left(\frac{f_1}{g_1},\frac{f_2}{g_2}\right):=\left(\frac{x_1^2+x_2}{x_2+1},\
	x_1^2-x_2+x_1\right)\Big|\ x \in \F\right\}.
\]
Clearly, this problem is in {\sf Case} I.
By checking if a given point is in the ellipse $\F$, it is easy to depict the
image of $\F$ in the way aforementioned,
which is shown in Figure \ref{fig::case1} (right).
Let the initial point be
$u^{(0)}=(-1,1)$ in Algorithm \ref{al::main}. The output is
$u^{\star}=u^{(2)}=(-0.2138,0.8319)$. These points and their images are marked
in Figure \ref{fig::case1}.
\begin{figure}
	\centering
\scalebox{0.45}{
	\includegraphics[trim=90 185 70 140,clip]{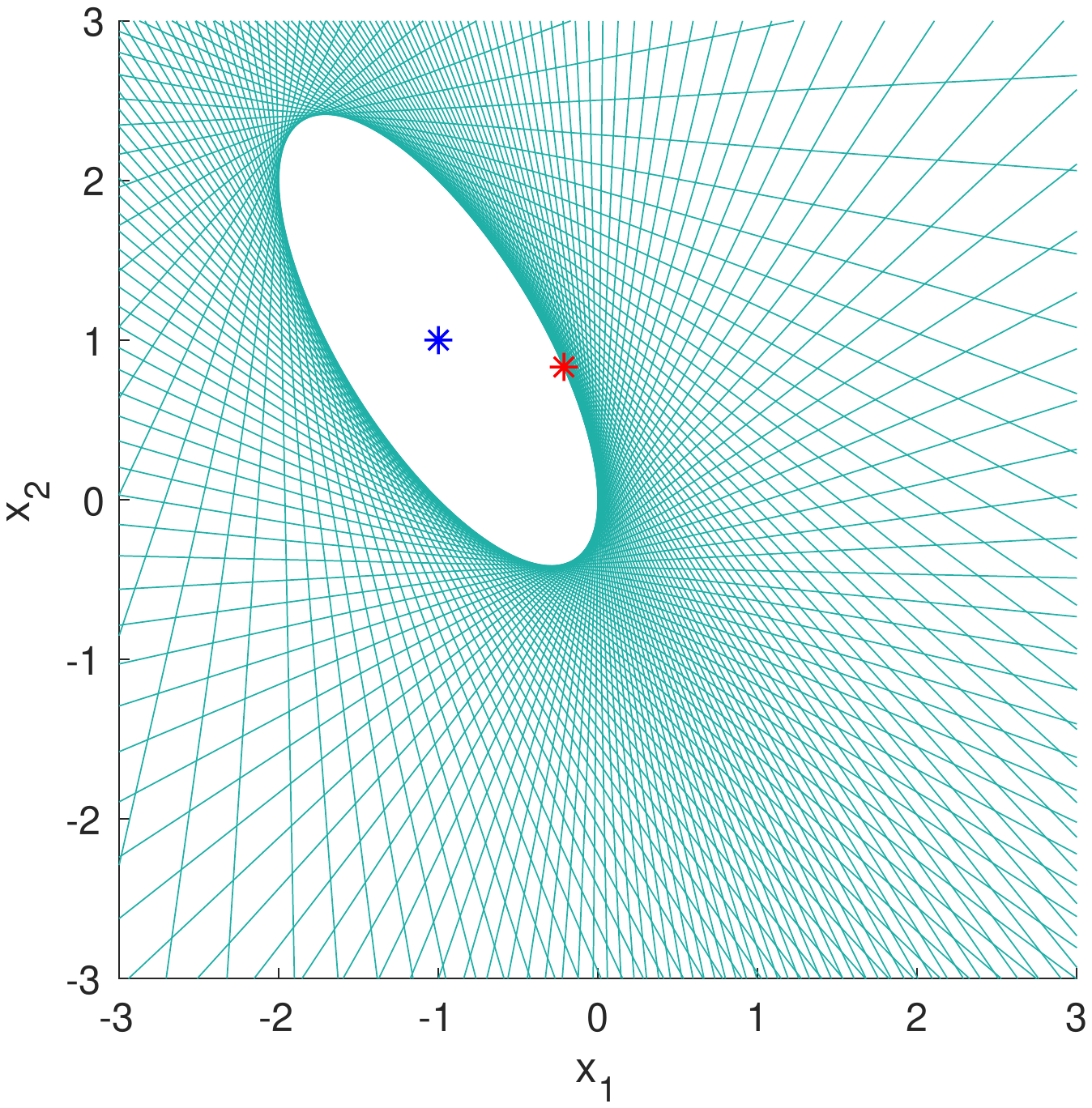}
}
\scalebox{0.45}{
	\includegraphics[trim=30 180 80 200,clip]{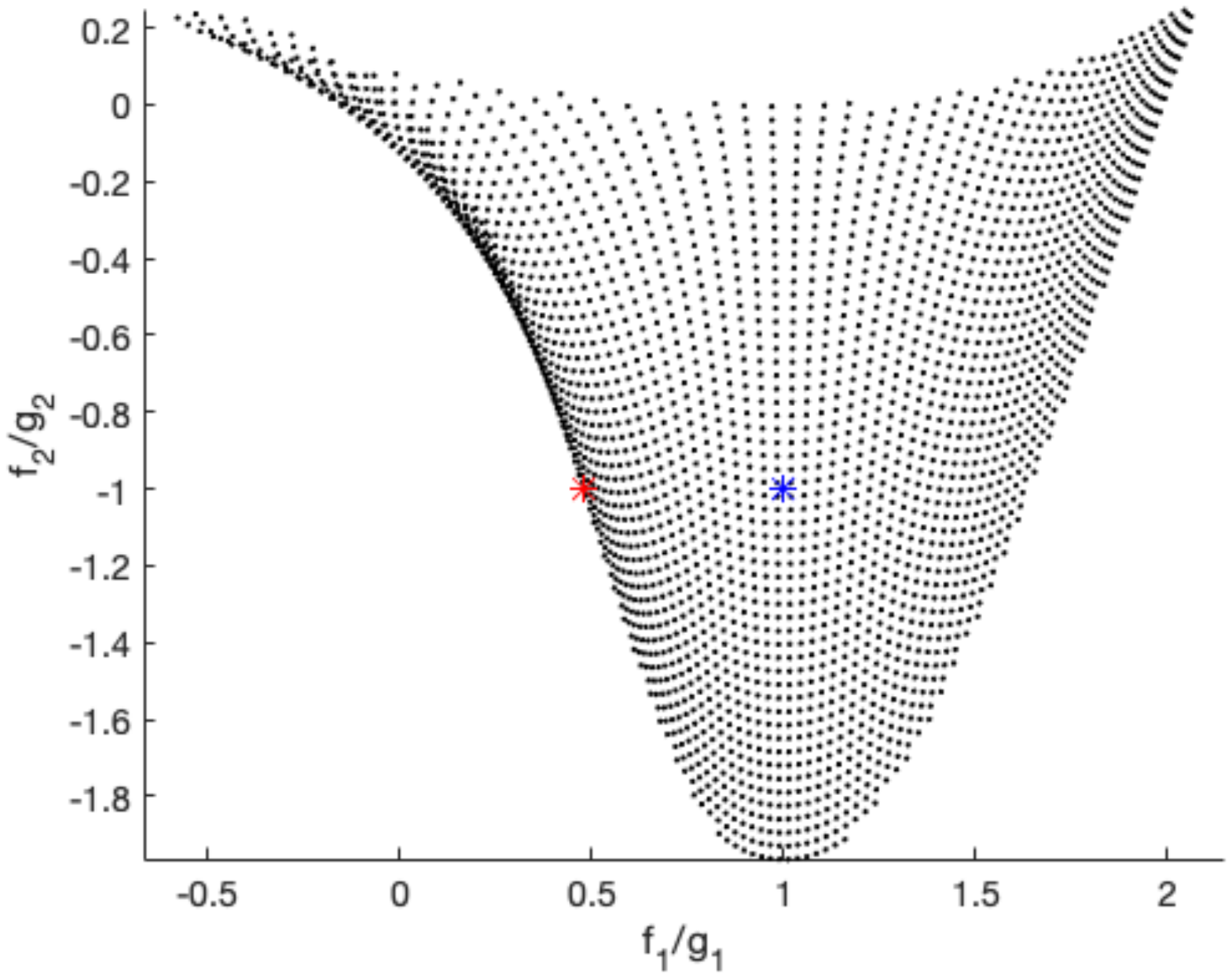}
}
	\caption{\label{fig::case1}The feasible set $\F$ (left) and its
	image (right) in the example of {\sf Case} I.}
\end{figure}

\vskip 5pt
\noindent{\sf Case} II:
	Consider the set
	\[
		\F=\{(x_1,x_2)\in\RR^2\mid p(x_1,x_2,y_1,y_2)\le
		0,\quad\forall y\in \Y\}
	\]
	where
	$p(x_1,x_2,y_1,y_2)=-1+x_1^2+x_2^2+(y_1-y_2)^2x_1x_2$ and
	\[
		\Y=\{(y_1,y_2)\in\RR^2\mid 1-y_1^2-y_2^2\ge 0\}.
	\]
	The set $\F$ is illustrated in Figure \ref{fig::case2} (left).
	The Hessian matrix of $p$ with respect to $x_1$ and $x_2$ is
	\[
		H=\left[\begin{array}{cc}
		2& (y_1-y_2)^2\\
(y_1-y_2)^2& 2
\end{array}\right]\qquad\text{with}\quad
\det(H)=4-(y_1-y_2)^4.
	\]
It is easy to see that $p(x_1,x_2,y_1,y_2)$ is s.o.s-convex in $(x_1,x_2)$
	for every $y\in \Y$.
	
	Consider the problem
	\[
	{\rm Min}_{\RR^{2}_+}\left\{\left(\frac{f_1}{g_1},\frac{f_2}{g_2}\right):=\left(\frac{x_2^2-x_1+1}{-x_1^2+2},\
	x_1^2+x_2+x_1\right)\Big|\ x \in \F\right\}.
\]
Clearly, this problem is in {\sf Case} II.
To depict the image of $\F$ in the aforementioned way,
we remark that $\F$ is in fact the area
enclosed by the lines $x_1+x_2=\pm 1$ and the unit circle. Hence, it
is easy to check whether a given point is in $\F$.
The image of $\F$ is shown in Figure \ref{fig::case2} (right).
Let the initial point be
$u^{(0)}=(0,1)$ in Algorithm \ref{al::main}. The output is
$u^{\star}=u^{(2)}=(0.6822,-0.1476)$.
These points and their images are marked in Figure \ref{fig::case2}.
\begin{figure}
	\centering
\scalebox{0.45}{
	\includegraphics[trim=80 200 80 200,clip]{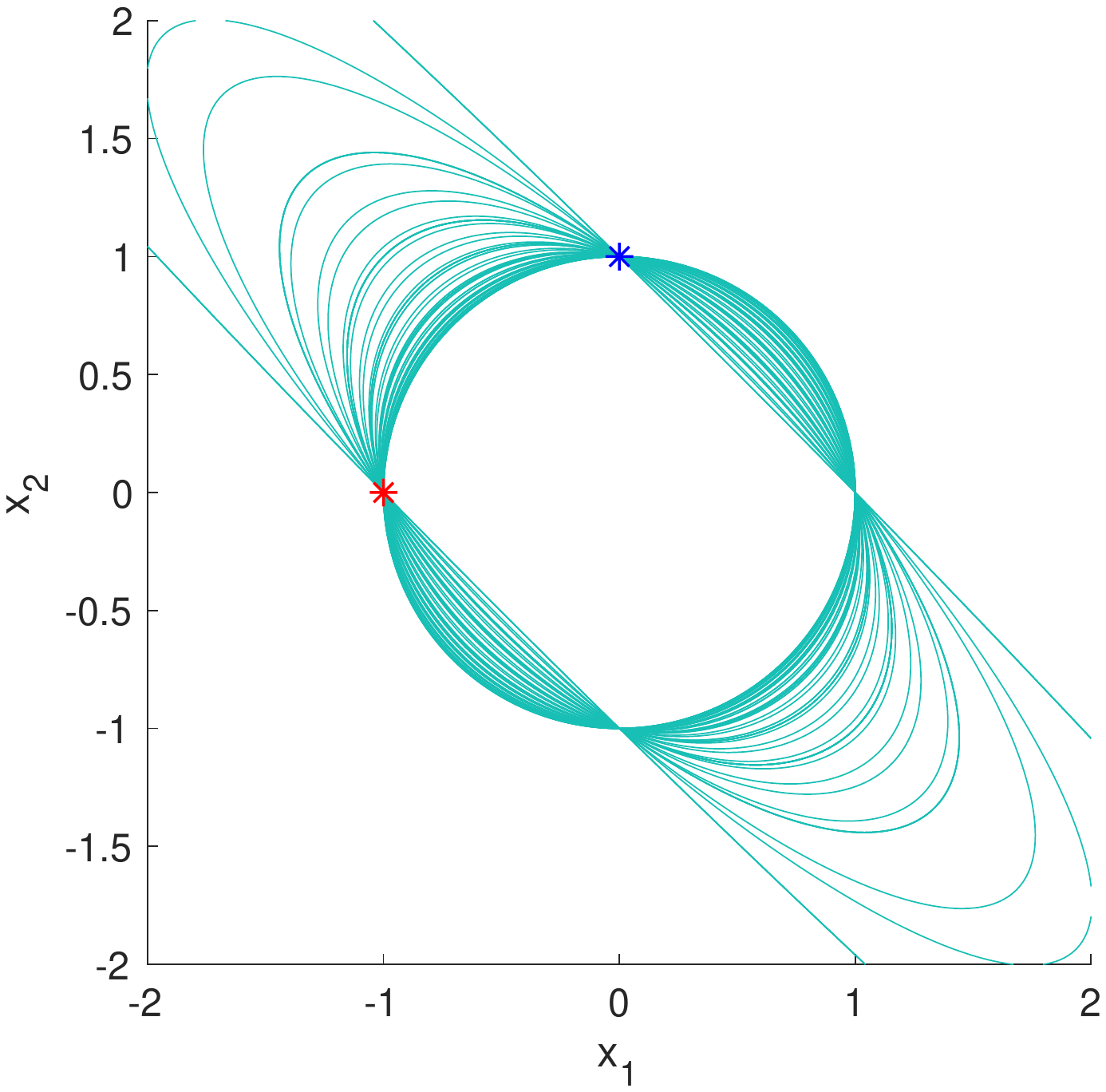}
}
\scalebox{0.45}{
	\includegraphics[trim=30 195 80 200,clip]{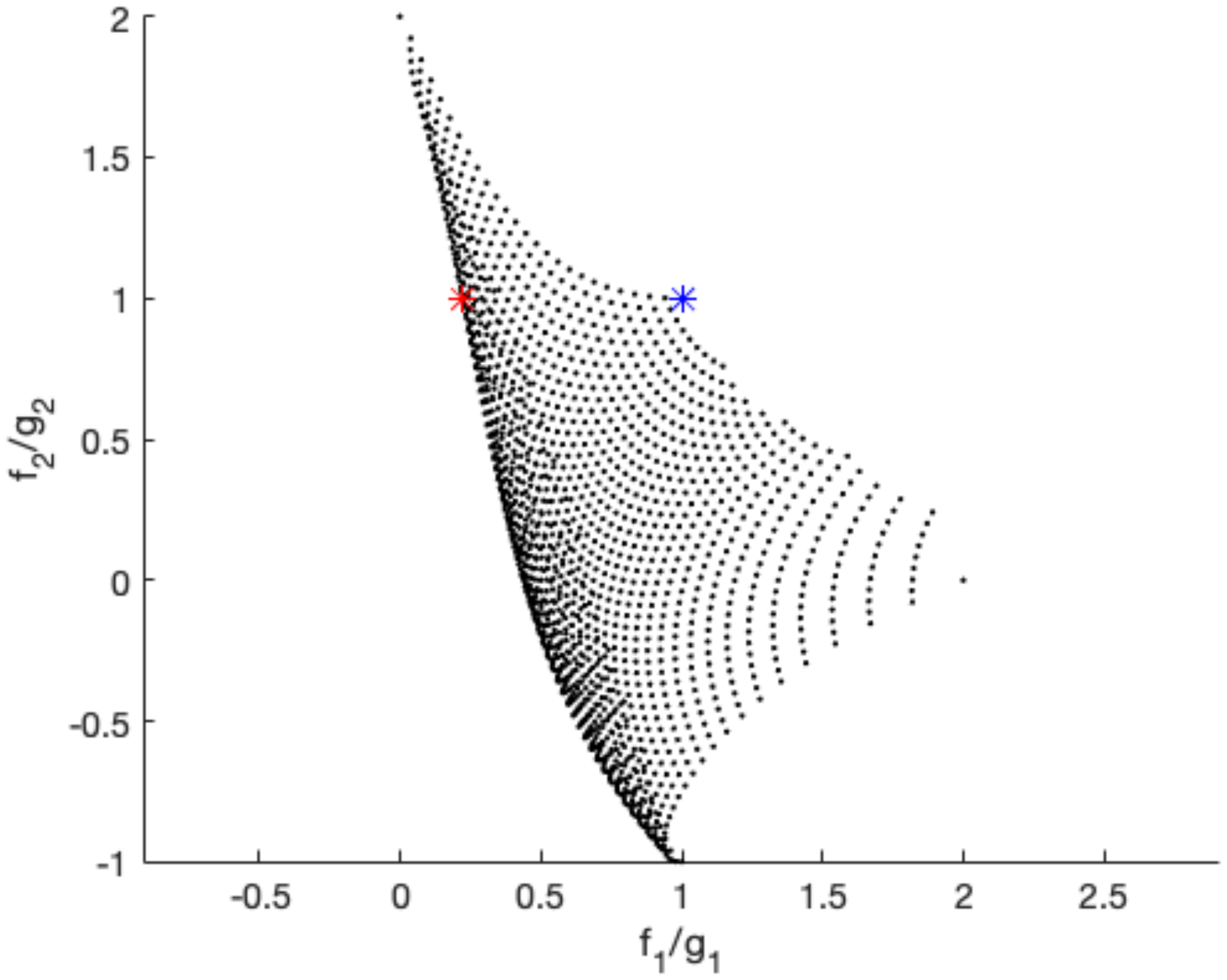}
}
	\caption{\label{fig::case2} The feasible set $\F$ (left) and its
	image (right) in the example of {\sf Case} II.}
\end{figure}

\vskip 5pt
\noindent{\sf Case} III:
Consider the polynomial $h_1(x_1,x_2,x_3)$ in \eqref{eq::nonsosconvex}
and let
	\[
		\F=\{(x_1,x_2)\in\RR^2\mid p(x_1,x_2,y_1)\le 0,\quad
		\forall y_1\in \Y\},
	\]
	where
	$p(x_1,x_2,y_1)=-1+h_1(x_1,x_2,1)/100-y_1x_1-y_1^2x_2$ and $\Y=[-1,1]$.
	Clearly, $p(x,y_1)$ is convex but not s.o.s-convex for every $y_1\in \Y$.
	We illustrate $\F$ in Figure \ref{fig::case3} (left).
	
	Consider the problem
	\begin{equation}\label{eq::ex3}
	 {\rm Min}_{\RR^{2}_+}\left\{\left(\frac{f_1}{g_1},\frac{f_2}{g_2}\right)
	 :=\left(\frac{x_1^2+x_2^2+1}{-x_1^2-x_2+3},\
	x_1^2+x_2^2-x_1+1\right)\Big|\ x \in \F\right\}.
\end{equation}
For a given point
$u\in\RR^2$, as $p(u_1,u_2,y_1)$ is a univariate quadratic function, it
is easy to check whether $-p(u_1,u_2,y_1)$ is nonnegative on $[-1,1]$
(i.e., whether $u\in\F$). Hence,
The image of $\F$ can be easily depicted in Figure \ref{fig::case3} (right).
Clearly, ($\mathrm{P}_1$) is in {\sf Case} 3. Hence, we only need to
solve ($\mathrm{P}_1$) to get an efficient solution by Proposition
\ref{prop::unique2} and Theorem \ref{th::correctness}. 
We let the initial point be $u^{(0)}=(-0.6,0.5)$ in Algorithm
\ref{al::main} and solve ($\mathrm{P}_1$) by the SDP relaxations
for {\sf Case} 3.
We set the first order $k=4$ and check if
the rank condition \eqref{eq::rk2} holds. If not, check the next order. 
We have  $\rank
\mathbf{M}_{3}(\mL_4^{\star})=\rank \mathbf{M}_{4}(\mL_4^{\star})=1$
(within a tolerance $<10^{-8}$) for a minimizer $\mL_4^{\star}$ of
of $r^{\dsdp}_4$, 
i.e., the rank condition \eqref{eq::rk2} holds for
$k'=4$. By Theorem \ref{th::sdp34} (ii), the point
$u^{\star}:=\frac{\mL_4^{\star}(x)}{\mL_4^{\star}(1)}=(0.000,-0.1623)$
is an efficient solution to
\eqref{eq::ex3}.
These points $u^{(0)}, u^{\star}$ and their images are marked in Figure
\ref{fig::case3}.
\begin{figure}
	\centering
\scalebox{0.45}{
	\includegraphics[trim=80 200 80 200,clip]{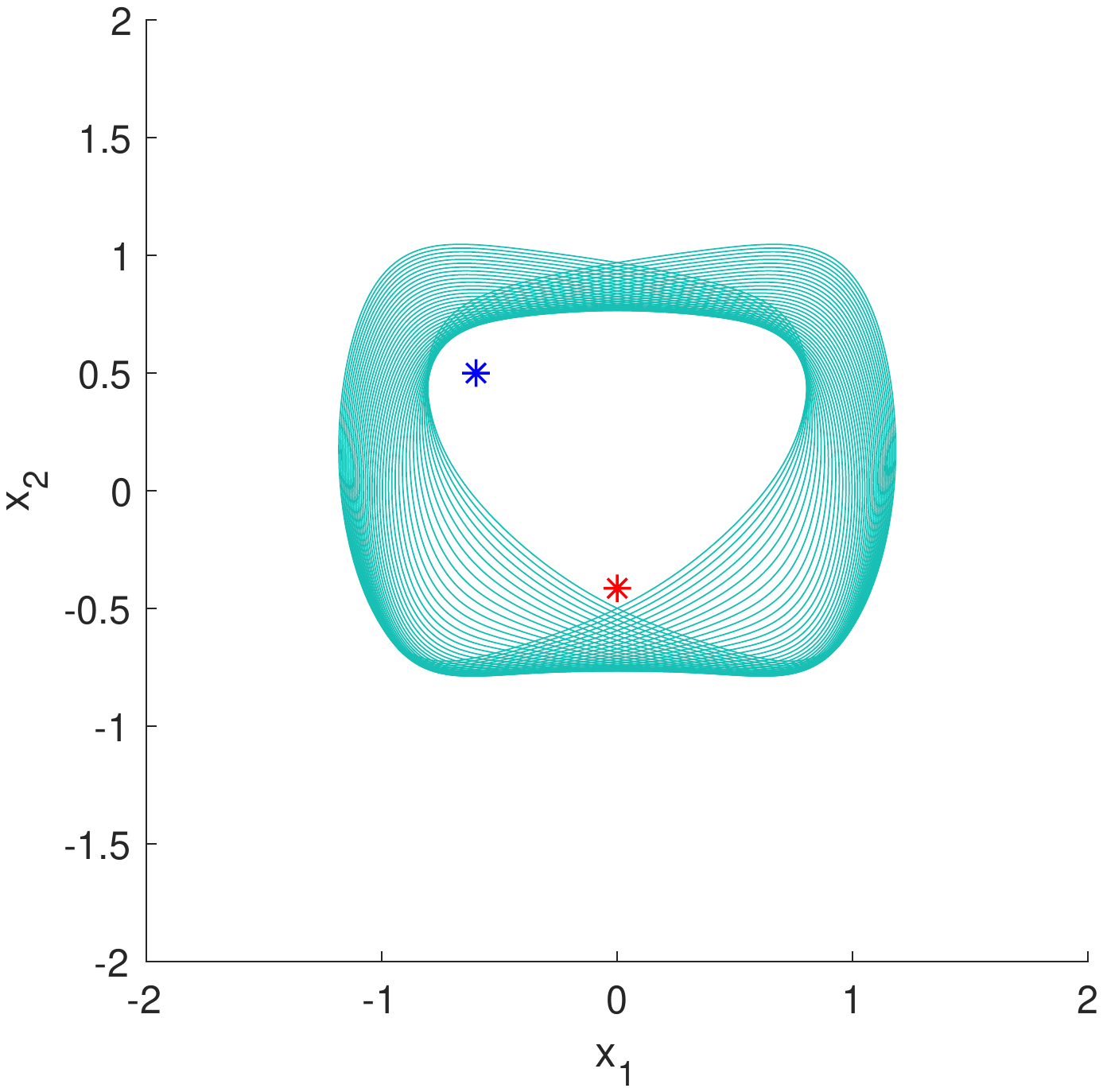}
}
\scalebox{0.45}{
	\includegraphics[trim=30 200 80 200,clip]{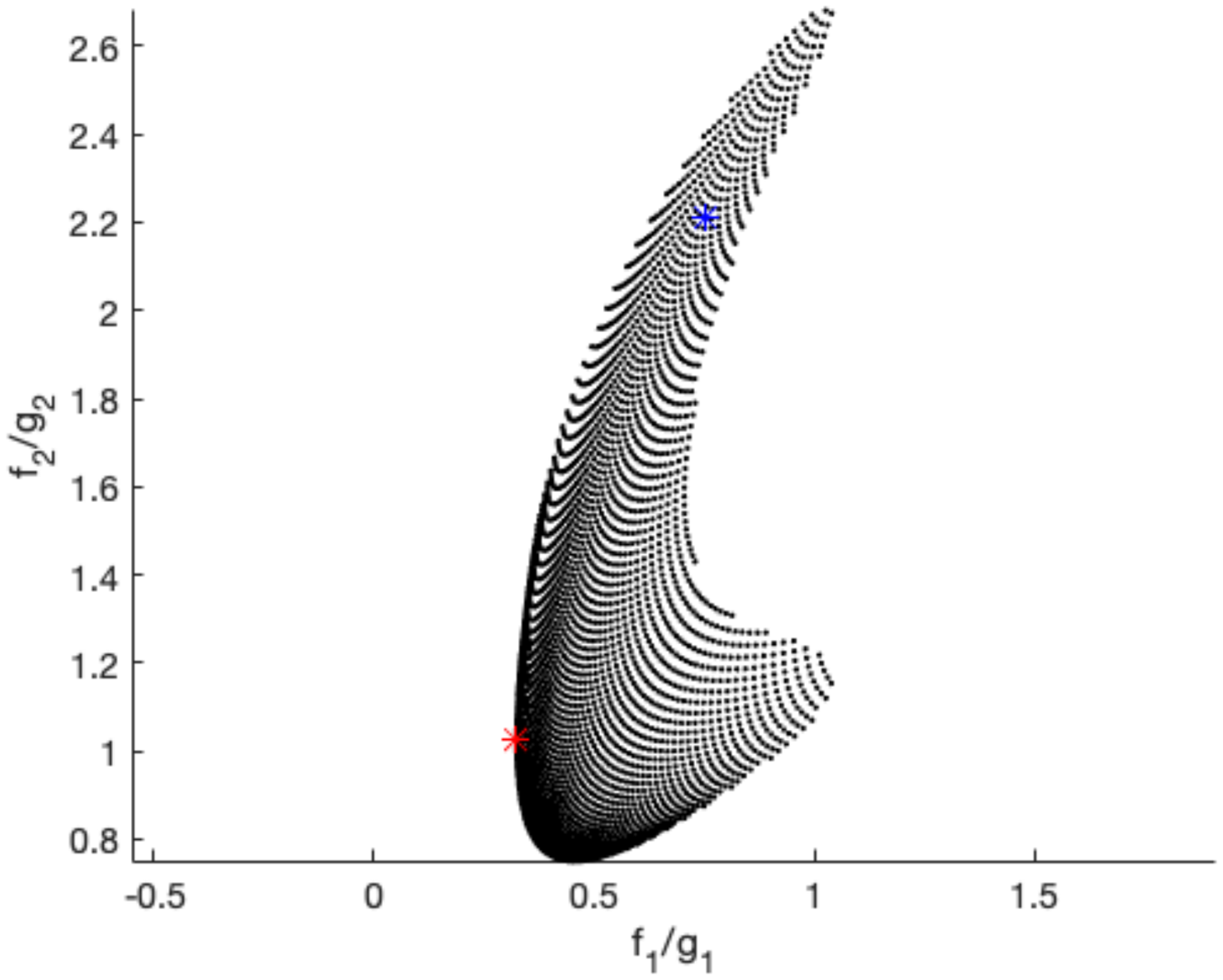}
}
	\caption{\label{fig::case3} The feasible set $\F$ (left) and its
	image (right) in the example of {\sf Case} III.}
\end{figure}

\vskip 5pt
\noindent{\sf Case} IV:
Let $h_1(x_1,x_2,x_3)$ be the polynomial in \eqref{eq::nonsosconvex} and
	\[
		\F=\{(x_1,x_2)\in\RR^2\mid p(x_1,x_2,y_1,y_2)\le 0,\quad
		\forall y\in \Y\},
	\]
	where
	$p(x_1,x_2,y_1,y_2)=(h_1(x_1,x_2,1)/100-1)-y_1y_2(x_1+x_2)$
	and
	\[
		\Y=\{(y_1,y_2)\in\RR^2\mid 1-y_1^2-y_2^2\ge 0\}.
	\]
	Clearly, $p(x,y)$ is convex but not s.o.s-convex for every $y\in
	\Y$.
	We illustrate $\F$ in Figure \ref{fig::case4} (left).
	
	Consider the problem
	\begin{equation}\label{eq::ex4}
	 {\rm Min}_{\RR^{2}_+}\left\{\left(\frac{f_1}{g_1},\frac{f_2}{g_2}\right)
	 :=\left(\frac{x_1^2+x_2^2+1}{-x_2^2+x_1+4},\
	\frac{x_1^2+x_2}{x_1+x_2+2}\right)\Big|\ x \in \F\right\}.
\end{equation}
To depict the image of $\F$ in the aforementioned way, we remark that $\F$
is in fact the area enclosed by the two curves
$p\left(x_1,x_2,\pm\frac{\sqrt{2}}{2},\pm\frac{\sqrt{2}}{2}\right)=0$. Hence, it
is easy to check whether a given point is in $\F$.
Then the image of $\F$ can be easily depicted in Figure \ref{fig::case4} (right).
Clearly, ($\mathrm{P}_1$) is in {\sf Case} 4. 
Again, we only need to solve ($\mathrm{P}_1$).
We let the initial point be $u^{(0)}=(-0.5,0.5)$ in Algorithm
\ref{al::main} and solve ($\mathrm{P}_1$) by the SDP relaxations for
{\sf Case} 4.
We check if the rank condition \eqref{eq::rk2}
holds for the order initialized from $4$.
Similarly to {\sf Case} III, when
$k=4$ and $k'=4$, the rank condition \eqref{eq::rk2} holds for  a
minimizer $\mL_4^{\star}$ of $r^{\dsdp}_4$.
By Theorem \ref{th::sdp34} (ii), the point
$u^{\star}:=\frac{\mL_4^{\star}(x)}{\mL_4^{\star}(1)}=(0.1231,0.000)$
is an efficient solution to
\eqref{eq::ex4}.
These points $u^{(0)}, u^{\star}$ and their images are marked in Figure
\ref{fig::case4}. \qed
\begin{figure}
	\centering
\scalebox{0.45}{
	\includegraphics[trim=80 200 80 200,clip]{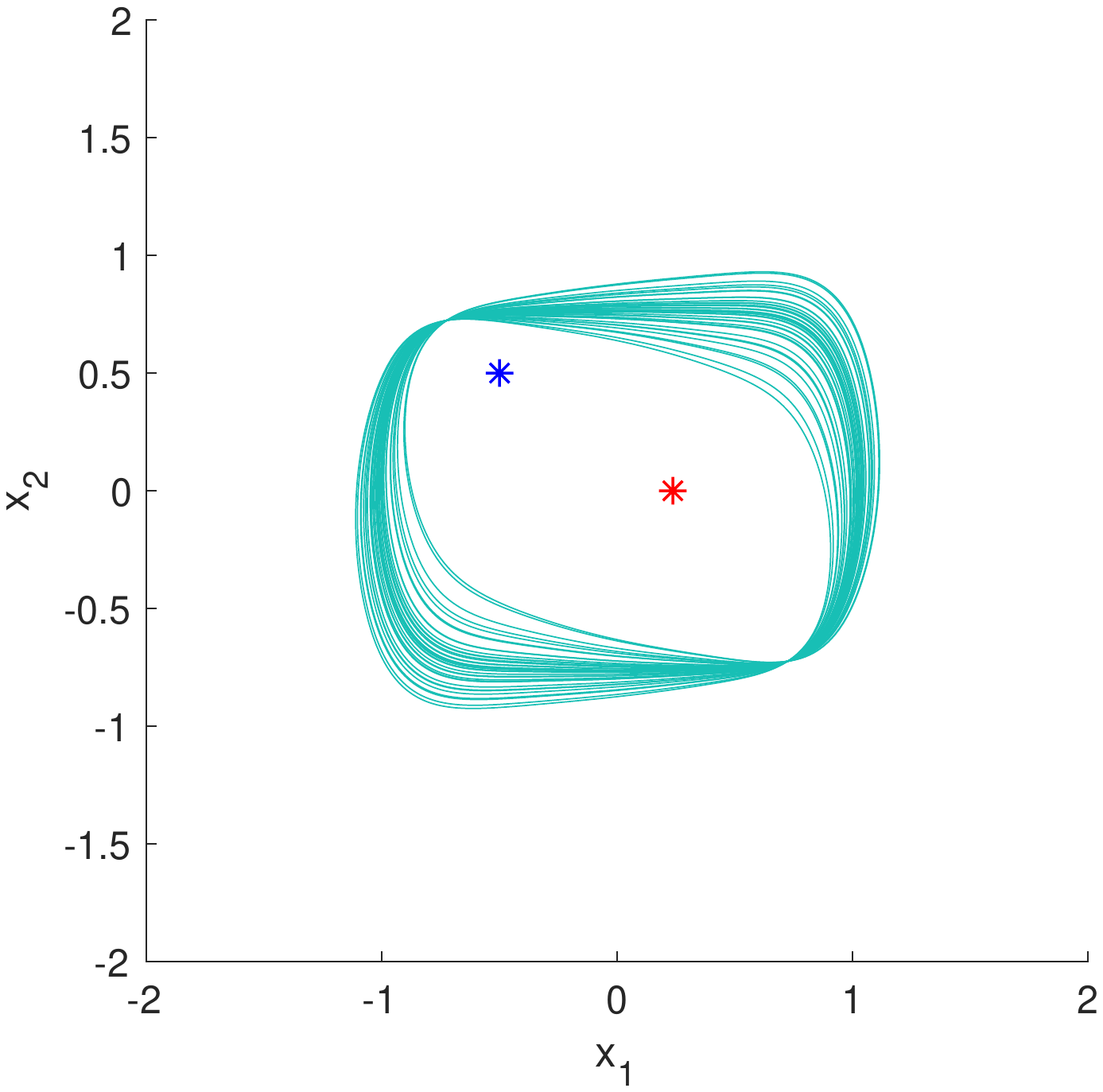}
}
\scalebox{0.45}{
	\includegraphics[trim=30 200 80 200,clip]{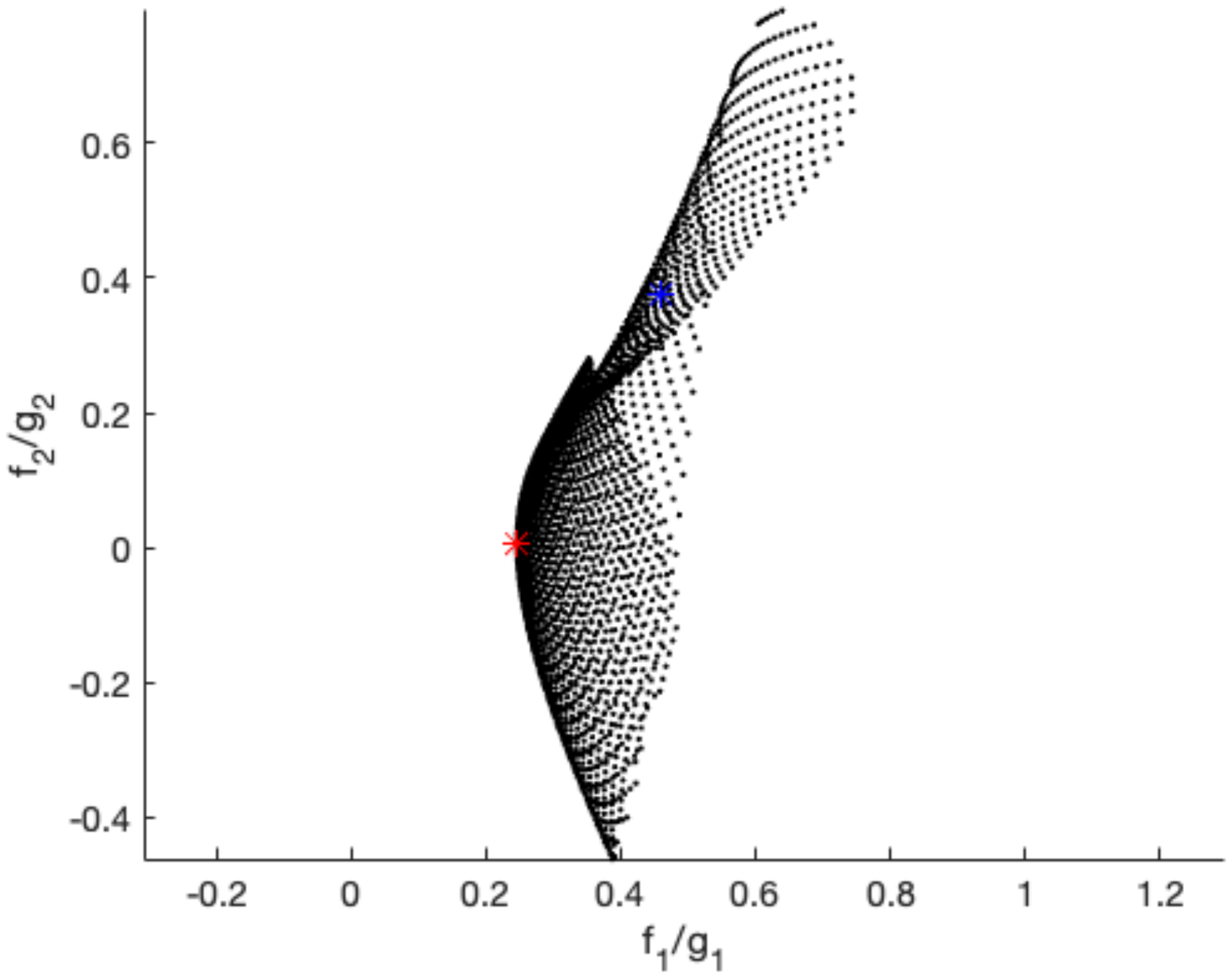}
}
	\caption{\label{fig::case4} The feasible set $\F$ (left) and its
	image (right) in the example of {\sf Case} IV.}
\end{figure}}
\end{example}

\section{Conclusions}\label{sec::conclusions}
We focus on solving a class of FSIPP problems with some
convexity/concavity assumption on the function data. We reformulate
the problem to a conic optimization
problem and provide a characteristic cone constraint qualification for
convex SIP problems to bring sum-of-squares structures in the
reformulation.
In this framework, we first present a hierarchy of SDP relaxations with
asymptotic convergence for the FSIPP problem whose index set is
defined by finitely many polynomial inequalities. 
Next, we study four cases of the FSIPP problems for which the SDP
relaxation is exact or has finite convergence and at least one
minimizer can be extracted. This approach is then applied to the four
corresponding multi-objective cases to find efficient solutions.

\section*{acknowledgements}
The authors are very grateful for the comments of two anonymous
referees which helped to improve the presentation.
The authors wish to thank Guoyin Li for
many helpful comments.
Feng Guo is supported by the Chinese National Natural Science
Foundation under grant 11571350, the Fundamental Research
Funds for the Central Universities.
Liguo Jiao is supported by Jiangsu Planned Projects for Postdoctoral Research
Funds 2019 (no. 2019K151).

\appendix
\section{}\label{appendix}

Consider the general convex semi-infinite programming problem
\begin{equation}\label{eq::csip}
	\left\{
		\begin{aligned}
			r^{\star}:=\inf_{x\in\RR^m}&\ h(x)\\
			\text{s.t.}&\ \psi_1(x)\le 0,\ldots,\psi_s(x)\le 0,\\
			&\ p(x,y)\le 0,\quad\forall\ y\in
			\Y\subset\RR^n,
		\end{aligned}\right.
\end{equation}
where $h(\cdot)$, $\psi_1(\cdot),\ldots,\psi_s(\cdot)$,
$p(\cdot,y):\RR^m\rightarrow\RR$ for
any $y\in\Y$, are continuous and convex functions (not necessarily
polynomials), $p(x,y):\RR^m\times\RR^n \rightarrow \RR$ is a lower
semicontinuous function such that $p(x,\cdot): \RR^n\rightarrow \RR$
is continuous for all $x \in \RR^m$, the index set $\Y$ is an
arbitrary compact subset in $\RR^n$. We denote by $\K$ the feasible
region of \eqref{eq::csip} and assume that $\K\neq\emptyset$.
Inspired by Jeyakumar and Li \cite{Jeyakumar2010},
we next provide a constraint qualification weaker than the Slater
condition for \eqref{eq::csip} to guarantee the strong duality and the
attachment of the solution in the dual problem.

Denote by $\mathcal{M}(\Y)$ the set of nonnegative measures supported
on $\Y$.  We first show that for all $\mu \in
\mathcal{M}(\Y),$
\[
	\Phi_{\mu}: x \mapsto \int_{\Y} p(x,y) d \mu(y)
\]
is a continuous and convex function. Indeed, it is clear that
this function always takes finite value due to the continuity
assumption of $p(x,\cdot)$  for all $x \in \RR^m$. Now, by Fatou's
lemma, for any $x^{(k)} \rightarrow x,$
\[
	\liminf_{k \rightarrow \infty} \int_{\Y} p(x^{(k)},y) d \mu(y) \ge \int_{\Y} p(x,y) d \mu(y).
\]
This shows that $\Phi_{\mu}$ is a lower semicontinuous function. Also, as
$p(\cdot,y)$ is convex and $\mu \in \mathcal{M}(\Y)$,
it is easy to see that $\Phi_{\mu}$ is also convex for all $\mu \in
\mathcal{M}(\Y)$. Thus, $\Phi_{\mu}$ is a proper lower semicontinuous convex
function which always takes finite value, and so, is continuous.

The Lagrangian dual of \eqref{eq::csip} reads
\begin{equation}\label{eq::dcsip2}
	\max_{\mu\in\mathcal{M}(\Y),\eta\in\RR_+^s} \inf_{x\in\RR^m} \left\{h(x)+\int_{\Y}
		p(x,y)d\mu(y)+\sum_{j=1}^s\eta_j\psi_j(x)\right\}.
\end{equation}

Recall the notation in \eqref{eq::cones},
\[
	\mathcal{C}_1=\bigcup_{\mu\in\mathcal{M}(\Y)}\epi\left(\int_{\Y}
	p(\cdot,y)d\mu(y)\right)^*\ \text{and}\
	\mathcal{C}_2=\bigcup_{\eta\in\RR_+^s}\epi\left(\sum_{j=1}^s\eta_j\psi_j\right)^*.
\]
We say that the semi-infinite characteristic cone constraint
qualification (SCCCQ) holds for $\K$ if
$\mathcal{C}_1+\mathcal{C}_2$ is closed.

\begin{proposition}\label{prop::cc}
	The set $\mathcal{C}_1+\mathcal{C}_2$ is a convex cone.
\end{proposition}
\begin{proof}
	As $\mathcal{C}_2$
is a convex cone due to \cite[Theorem 2.123]{DD}, we only need to
prove that $\mathcal{C}_1$
is a convex cone.

We first prove that $\mathcal{C}_1$
is a cone.
It is clear that $(0,0)\in\mathcal{C}_1$.
Let $\lambda>0$ and $(\xi,\alpha)\in\mathcal{C}_1$
Then, there exists $\mu'\in\mathcal{M}(\Y)$ such that
$(\xi,\alpha)\in\epi\left(\int_{\Y} p(\cdot,y)d\mu'(y)\right)^*$. Let
$\td{\mu}=\lambda\mu'\in\mathcal{M}(\Y)$. As $\Phi_{\mu'}$ is
continuous and convex, by \cite[Theorem 2.123 (iv)]{DD},
\[
	\lambda(\xi,\alpha)\in\lambda\epi\left(\int_{\Y}
	p(\cdot,y)d\mu'(y)\right)^*=\epi\left(\lambda\int_{\Y}
	p(\cdot,y)d\mu'(y)\right)^*=\epi\left(\int_{\Y}
	p(\cdot,y)d\td{\mu}(y)\right)^*.
\]
Hence, $\lambda(\xi,\alpha)\in\mathcal{C}_1$.

Now it suffices to prove that
$\co(\mathcal{C}_1)\subseteq\mathcal{C}_1$.
Let $(\xi,\alpha)\in\co(\mathcal{C}_1$).
As $\mathcal{C}_1$
is a cone in $\RR^{m+1}$, from the Carathedory theorem, there
exist $(\xi_{\ell},\alpha_{\ell})\in\mathcal{C}_1$,
${\ell}=1,\ldots,m+1$, such that
$(\xi,\alpha)=\sum_{{\ell}=1}^{m+1}(\xi_{\ell},\alpha_{\ell})$. For each
${\ell}=1,\ldots,m+1$, there exists $\mu_{\ell}\in\mathcal{M}(\Y)$ such that
$(\xi_{\ell},\alpha_{\ell})\in\epi\left(\int_{\Y} p(\cdot,y)d\mu_{\ell}(y)\right)^*$. Note
that $\Phi_{\mu_{\ell}}$ is continuous for each ${\ell}=1,\ldots,m+1$. Let
$\hat{\mu}=\sum_{{\ell}=1}^{m+1}\mu_{\ell}\in\mathcal{M}(\Y)$, then by
\cite[Theorem 2.123 (i) and Proposition 2.124]{DD},
\[
	\begin{aligned}
		(\xi,\alpha)=\sum_{{\ell}=1}^{m+1}(\xi_{\ell},\alpha_{\ell})\in\sum_{{\ell}=1}^{m+1}\epi\left(\int_{\Y}
		p(\cdot,y)d\mu_{\ell}(y)\right)^*&=\epi\left(\sum_{{\ell}=1}^{m+1}\int_{\Y} p(\cdot,y)d\mu_{\ell}(y)\right)^*\\
		&=\epi\left(\int_{\Y}
		p(\cdot,y)d\hat{\mu}(y)\right)^*\subset\mathcal{C}_1.
	\end{aligned}
\]
The proof is completed.%\qed
\end{proof}

\begin{theorem}\label{th::alt}
	Exactly one of the following two statements holds$:$
	\begin{enumerate}[\upshape (i)]
		\item $(\exists x\in\RR^m)\ h(x)<0,\ \psi_j(x) \le 0,\
			j=1,\ldots,s, \ p(x,y)\le 0, \forall\ y\in\Y;$
		\item $(0,0)\in\epi\ h^*+\cl(\mathcal{C}_1+\mathcal{C}_2).$
	\end{enumerate}
\end{theorem}
\begin{proof}
Let
\[
	\K_1:=\left\{x\in\RR^m : \int_{\Y} p(x,y)d\mu(y)\le 0,\ \forall
		\mu\in\mathcal{M}(\Y)\right\},
\]
and
\[
	\K_2:=\left\{x\in\RR^m : \sum_{j=1}^s\eta_j\psi_j(x)\le 0,\
\forall \eta\in\RR_+^s\right\}.
\]
It is easy to see that $\K=\K_1\cap\K_2$ and the indicator
functions of $\K_1$ and $\K_2$ are
\[
	\delta_{\K_1}(x)=\sup_{\mu\in\mathcal{M}(\Y)}\int_{\Y}
	p(x,y)d\mu(y)\quad\text{and}\quad
	\delta_{\K_2}(x)=\sup_{\eta\in\RR_+^s} \sum_{j=1}^s\eta_j\psi_j(x).
\]
By Proposition \ref{prop::cc} and \cite[Lemma 2.2]{Li_Ng}, it holds
that
\[
	\epi\
	(\delta_{\K_1})^*=\cl(\mathcal{C}_1)\quad\text{and}\quad\epi\
	(\delta_{\K_2})^*=\cl(\mathcal{C}_2).
\]
Now, we show that [not (i)] is equivalent to [(ii)]. In fact,
\[
	\begin{aligned}
		\text{[not (i)]}\ &\Leftrightarrow\ h(x)\ge 0,\ \forall
		x\in\K_1\cap\K_2  \\
&\Leftrightarrow\
\inf_{x\in\RR^m}\{h(x)+\delta_{\K_1}(x)+\delta_{\K_2}(x)\}\ge 0\\
&\Leftrightarrow\ (0,0)\in\epi (h+\delta_{\K_1}+\delta_{\K_2})^*
	\end{aligned}
\]
By the continuity of $h$ and \cite[Theorem 2.123 (i)]{DD}, we have
\[
	\begin{aligned}
		\epi\ (h+\delta_{\K_1}+\delta_{\K_2})^* &=\epi\ h^*+\epi\ (\delta_{\K_1}+\delta_{\K_2})^* \\
&=\epi\ h^*+ {\rm cl}(\epi\ (\delta_{\K_1})^*+\epi\ (\delta_{\K_2})^*) \\
&=\epi\ h^*+{\rm cl}(\cl(\mathcal{C}_1)+\cl(\mathcal{C}_2)) \\
&=\epi\ h^*+{\rm cl}(\mathcal{C}_1+\mathcal{C}_2)).
	\end{aligned}
\]
Hence, the conclusion follows.
\end{proof}

\begin{theorem}\label{th::cq}
	Suppose that the {\rm SCCCQ} holds for \eqref{eq::csip}$,$
	then there exist $\mu^{\star}\in\mathcal{M}(\Y)$ and
$\eta^{\star}\in\RR_+^s$ such that
\[
	r^{\star}=\inf_{x\in\RR^m}\left\{h(x)+\int_{\Y}
		p(x,y)d\mu^{\star}(y)+\sum_{j=1}^s\eta^{\star}_j\psi_j(x)\right\},
\]
where $r^{\star}$ is the optimal value of \eqref{eq::csip}.
\end{theorem}
\begin{proof}
From the weak duality, we have
\[
r^{\star} \ge \max_{\mu\in\mathcal{M}(\Y),\eta\in\RR_+^s} \inf_{x\in\RR^m}
\left\{h(x)+\int_{\Y}
p(x,y)d\mu(y)+\sum_{j=1}^s\eta_j\psi_j(x)\right\}.
\]
As we assume that $\K\neq\emptyset$, $r^{\star}>-\infty$.
Applying Theorem \ref{th::alt} with $h$ replaced by
$\overline{h}$ where $\overline{h}(x)=h(x)-r^{\star}$ for all $x \in \RR^m$,
and making use of the SCCCQ, one has
\begin{eqnarray*}
	(0,0) &\in & \epi\ \overline{h}^*+\bigcup_{\mu\in\mathcal{M}(\Y)}\epi\left(\int_{\Y}
	p(\cdot,y)d\mu(y)\right)^*+\bigcup_{\eta\in\RR_+^s}\epi\left(\sum_{j=1}^s\eta_j\psi_j\right)^* \\
	& = & \epi\ h^*+ (0,r^{\star})+ \bigcup_{\mu\in\mathcal{M}(\Y)}\epi\left(\int_{\Y}
	p(\cdot,y)d\mu(y)\right)^*+\bigcup_{\eta\in\RR_+^s}\epi\left(\sum_{j=1}^s\eta_j\psi_j\right)^* \ .
\end{eqnarray*}
Then, there exist $(\xi,\alpha)\in \epi\ h^*$,
$\mu^{\star}\in\mathcal{M}(\Y)$, $(\tau,\beta)\in\epi\left(\int_{\Y}
	p(\cdot,y)d\mu^{\star}(y)\right)^*$, $\eta^{\star}\in\RR_+^s$,
	$(\zeta,\gamma)\in\epi\left(\sum_{j=1}^s\eta_j^{\star}\psi_j\right)^*$ such that
	$(\xi,\alpha)+(\tau,\beta)+(\zeta,\gamma)=(0,-r^{\star})$. Then, for every $x\in\RR^m$,
	\[
		\begin{aligned}
			&-h(x)-\int_{\Y}
			p(\cdot,y)d\mu^{\star}(y)-\sum_{j=1}^s\eta^{\star}_j\psi_j(x)\\
			&=\langle \xi, x\rangle-h(x)+\langle
			\tau, x\rangle-\int_{\Y}
			p(\cdot,y)d\mu^{\star}(y)+\langle\zeta,x\rangle-\sum_{j=1}^s\eta^{\star}_j\psi_j(x)\\
		&\le h^*(\xi)+\left(\int_{\Y}
		p(\cdot,y)d\mu^{\star}(y)\right)^*(\tau)+\left(\sum_{j=1}^s\eta^{\star}_j\psi_j\right)^*(\zeta)\\
		&\le \alpha+\beta+\gamma=-r^{\star}.
		\end{aligned}
	\]
	Then the conclusion follows by the weak duality.%\qed
\end{proof}

Recall that the {\rm Slater} condition holds for \eqref{eq::csip}
if there exists $u\in\RR^m$ such that $p(u,y)<0$ for all $y\in \Y$ and
$\psi_j(u)<0$ for all $j=1,\ldots,s.$ We show that the {\rm Slater}
condition can guarantee the SCCCQ condition.

\begin{proposition}\label{prop::cq2}
If the {\rm Slater} condition holds for \eqref{eq::csip}, then
$\mathcal{C}_1+\mathcal{C}_2$ is closed.
\end{proposition}
\begin{proof}
	Let $\left(w^{(k)},v^{(k)}\right)\in\mathcal{C}_1+\mathcal{C}_2$ such that
	$\left(w^{(k)},v^{(k)}\right)\rightarrow (w,v)$ and we show that
	$(w,v)\in\mathcal{C}_1+\mathcal{C}_2$. For each $k\in\N$, there
	exist $\left(w^{(k,1)},v^{(k,1)}\right)\in\mathcal{C}_1$ and
	$\left(w^{(k,2)},v^{(k,2)}\right)\in\mathcal{C}_2$ such that
	\[
	\left(w^{(k)},v^{(k)}\right)=\left(w^{(k,1)},v^{(k,1)}\right)+\left(w^{(k,2)},v^{(k,2)}\right).
\]
	Then, for each $k\in\N$, there exists a measure
	$\mu^{(k)}\in\mathcal{M}(\Y)$ and $\eta^{(k)}\in\RR_+^s$ such that
	for any $x\in\RR^m$,
	\begin{equation}\label{eq::epi1}
		v^{(k,1)}\ge \langle w^{(k,1)}, x\rangle-\int_{\Y}
			p(x,y)d\mu^{(k)}(y),
	\end{equation}
	and
	\begin{equation}\label{eq::epi2}
		v^{(k,2)}\ge \langle w^{(k,2)},
			x\rangle-\sum_{j=1}^s\eta_j^{(k)}\psi_j(x).
	\end{equation}
	Therefore, for any $x\in\RR^m$,
	\begin{equation}\label{eq::epi3}
			v^{(k)}\ge \langle w^{(k)}, x\rangle-\int_{\Y}
			p(x,y)d\mu^{(k)}(y)-\sum_{j=1}^s\eta_j^{(k)}\psi_j(x).
	\end{equation}
	Without loss of generality, we may assume that
	$(w,v)\not\in\{0\}\times\RR_+$ since
	$\{0\}\times\RR_+\subset\mathcal{C}_1+\mathcal{C}_2$. Hence, for
	each $k\in\N$, without loss of generality, we may assume that $\int_{\Y}
	d\mu^{(k)}(y)+\sum_{j=1}^s\eta_j^{(k)}>0$ and let
	\[
		\td{\mu}^{(k)}=\frac{\mu^{(k)}}{\int_{\Y}
		d\mu^{(k)}(y)+\sum_{j=1}^s\eta_j^{(k)}},\quad\quad
		\td{\eta}^{(k)}=\frac{\eta^{(k)}}{\int_{\Y}
		d\mu^{(k)}(y)+\sum_{j=1}^s\eta_j^{(k)}}.
	\]
	Then, passing to
	subsequences if necessary, we may assume that there are a
	measure $\nu\in\mathcal{M}(\Y)$ and a point $\xi\in\RR_+^s$ such that
	the sequence $\{\td{\mu}^{(k)}\}$ is weakly convergent to $\nu$
	by Prohorov's theorem (c.f. \cite[Theorem 8.6.2]{Bogachev}) and
	the sequence $\{\td{\eta}^{(k)}\}$ is convergent to $\xi$.
	We claim that both $\int_{\Y} d\mu^{(k)}(y)$ and
	$\sum_{j=1}^s\eta_j^{(k)}$ are bounded as $k\rightarrow\infty$. If
	it is not the case, then dividing both
	sides of \eqref{eq::epi3} by $\int_{\Y}
		d\mu^{(k)}(y)+\sum_{j=1}^s\eta_j^{(k)}$ and letting $k$ tend to $\infty$ yealds
	\[
		0\ge -\int_{\Y} p(x,y)d\nu(y)-\sum_{j=1}^s\xi_j\psi_j(x),\quad \forall x\in\RR^m.
	\]
	Recall that $p(x,\cdot): \RR^n\rightarrow \RR$ is continuous for
	all $x \in \RR^m$.
	As the Slater condition holds and $\Y$ is compact, there exist a
	point $u\in\RR^m$ and a constant $c<0$ such that
	\[
		\int_{\Y} p(u,y)d\nu(y)+\sum_{j=1}^s\xi_j\psi_j(u)\le c<0,
	\]
	a contradiction.
	Then, passing to subsequences if necessary, we may assume that there is a measure
	$\tau\in\mathcal{M}(\Y)$ and a point $\chi\in\RR_+^s$ such that
	the sequence $\{\mu^{(k)}\}$ is weakly convergent to $\tau$ by
	Prohorov's theorem again and the sequence $\{\eta^{(k)}\}$ is
	convergent to $\chi$. Letting $k$ tend to $\infty$ in
	\eqref{eq::epi3} yealds that for any
	$x\in\RR^m$
	\[
		v\ge \langle w, x\rangle-\int_{\Y}
		p(x,y)d\tau(y)-\sum_{j=1}^s\chi_j\psi_j(x),
	\]
	i.e., $(w,v)\in\epi\left(\int_{\Y}
	p(\cdot,y)d\tau(y)+\sum_{j=1}^s\chi_j\psi_j\right)^*$. As both
	$\int_{\Y} p(\cdot,y)d\tau(y)$ and $\sum_{j=1}^s\chi_j\psi_j$ are
	continuous on $\RR^m$, we have
	\[
\epi\left(\int_{\Y}
p(\cdot,y)d\tau(y)+\sum_{j=1}^s\chi_j\psi_j\right)^*=\epi\left(\int_{\Y}
p(\cdot,y)d\tau(y)\right)^*+\epi\left(\sum_{j=1}^s\chi_j\psi_j\right)^*
\subset\mathcal{C}_1+\mathcal{C}_2.
	\]
	 Therefore, the conclusion follows.
\end{proof}

%\bibliographystyle{abbrv}
%\bibliography{mfsipp}

\end{document}